\documentclass{amsart}
\usepackage[utf8]{inputenc}
\usepackage[T1]{fontenc}
\usepackage[english]{babel}
\usepackage{fancyhdr}
\usepackage{graphicx}
\usepackage{float}
\usepackage{hyperref}
\usepackage{todonotes}
\usepackage{amsmath}
\usepackage{amssymb}
\usepackage{amsfonts}
\usepackage{amsthm}
\usepackage{mathrsfs}
\usepackage{verbatim}
\usepackage{breqn}
\usepackage{soul}
\usepackage{tikz-cd}
\usepackage{mathtools}
\usepackage{extpfeil}
\usepackage{caption}

\usepackage{extpfeil}
\usetikzlibrary{decorations.pathreplacing,angles,quotes}

\newtheorem{mainthm}{Theorem}

\newtheorem{thm}{Theorem}
\numberwithin{thm}{subsection}
\numberwithin{equation}{section}

\newtheorem{lem}[thm]{Lemma}
\newtheorem{remark}[thm]{Remark}
\newtheorem{prop}[thm]{Proposition}
\newtheorem{cor}[thm]{Corollary}
\theoremstyle{definition}
\newtheorem{eg}[thm]{Example}

\newcommand{\C}{\mathbb{C}}
\newcommand{\Ct}{\mathbb{C}{(\!(\!\hspace{0.7pt}t\hspace{0.7pt}\!)\!)}}
\newcommand{\Cu}{\mathbb{C}{(\!(\!\hspace{0.7pt}u\hspace{0.7pt}\!)\!)}}
\newcommand{\D}{\mathbb{D}}
\newcommand{\Q}{\mathbb{Q}}
\newcommand{\sing}{\mathrm{sing}}
\newcommand{\X}{\mathscr{X}}
\renewcommand{\P}{\mathbb{P}}
\renewcommand{\O}{\mathcal{O}}
\newcommand{\an}{\mathrm{an}}
\newcommand{\An}{\mathrm{An}}
\newcommand{\R}{\mathbb{R}}
\newcommand{\Log}{\mathrm{Log}}
\newcommand{\hyb}{\mathrm{hyb}}
\newcommand{\N}{\mathbb{N}}

\newcommand{\ord}{\mathrm{ord}}
\renewcommand{\L}{\mathcal{L}}
\newcommand{\Res}{\mathrm{Res}}

\newcommand{\supp}{\mathrm{supp}}
\newcommand{\Spec}{\mathrm{Spec \ }}
\newcommand{\Sk}{\mathrm{Sk}}
\newcommand{\Hom}{\mathrm{Hom}}
\newcommand{\Z}{\mathbb{Z}}
\renewcommand{\div}{\mathrm{div}}
\newcommand{\red}{\mathrm{red}}

\newcommand{\Y}{\mathscr{Y}}
\newcommand{\M}{\mathcal{M}}
\renewcommand{\tilde}{\widetilde}
\newcommand{\I}{\mathcal{I}}
\newcommand{\reg}{\mathrm{reg}}

\title[Non-Archimedean degenerations of volume forms]{Convergence of volume forms on a family of log-Calabi-Yau varieties to a non-Archimedean measure}
\author{Sanal Shivaprasad}
\date{\today}

\begin{document}

\begin{abstract}
  We study the convergence of volume forms on a degenerating holomorphic family of log-Calabi-Yau varieties to a non-Archimedean measure, extending a result of Boucksom and Jonsson. More precisely, let $(X,B)$ be a holomorphic family of sub log canonical, log-Calabi-Yau complex varieties parameterized by the punctured unit disk. Let $\eta$ be a meromorphic volume form on $X$ with poles along $B$. We show that the (possibly infinite) measures induced by the restriction of the $\eta$ to a fiber converge to a measure on the Berkovich analytification as we approach the puncture. The convergence takes place on a hybrid space, which is obtained by filling in the space $X \setminus B$ with the aforementioned Berkovich space over the puncture. 
\end{abstract}

\maketitle

\section{Introduction}
If $Y$ is an irreducible, normal and compact complex analytic space such that we have a top-dimensional meromorphic form $\eta$ on the smooth locus, $Y^\reg$, and a divisor $D \subset Y$ such that $\eta$ is holomorphic and does not vanish on $Y^\reg \setminus D$ and has poles given exactly by $D$, then the pair $(Y,D)$ is called log-Calabi-Yau. Any two such forms $\eta$ and $\eta'$ on $Y^\reg$ which have poles given by $D$ will be unique up to a scalar factor. The form $\eta$ gives rise to the volume form $i^{(\dim Y)^2}\eta \wedge \overline{\eta}$ on $Y^\reg \setminus D$, and thus a positive measure on $Y$. For a log-Calabi-Yau $(Y,D)$, this measure is unique up to scaling. Note that locally near $D$ and $Y^\sing$, it is possible for the mass to be infinite. When $D = 0$, $Y$ is said to be Calabi-Yau.

Families of (log-)Calabi-Yau varieties appear in many settings, for example in geometry and mirror symmetry \cite{Bat93}. It would be natural to ask how this canonical measure varies along families of log-Calabi-Yau varieties. The main motivation for our problem comes from \cite{BJ}, where Boucksom and Jonsson studied this canonical measure along  families of Calabi-Yau varieties. We extend some of the results to families of log-Calabi-Yau varieties.

Let $X \to \D^*$ be a proper flat family of irreducible normal complex analytic spaces. Let $B \subset X$ be a $\Q$-Weil divisor such that $K_{X/\D^*} + B$ is $\Q$-Cartier and is $\Q$-linearly equivalent to 0. Then, $(X_t,B|_{X_t})$ is log-Calabi-Yau for $| t | \ll  1$. Using the above recipe, we can obtain measures $\mu_t$ on each of the fibers $X_t$ for $t \ll 1$. Two such families $\mu_t$ and $\mu'_t$ would differ by a factor of $|h(t)|^{2/m}$, where $h$ is a holomorphic function on $\D^*$ and $m$ is an integer.  Our goal is to understand if the measure $\mu_t$  converge in some sense as $t \to 0$.

meromorphic form on $X^{\reg}_t$ with poles along $B_t \cap X_t^\reg$ which gives rise to a volume form on $X^{\reg}_t \setminus B_t$, and hence a (possibly infinite) positive  measure, $\mu_t$ on $X_t \setminus B_t$. Even if $m \neq 1$, we can still obtain a measure $\mu_t$ on $X_t \setminus B_t$ using $\eta$ (For more details, see Section \ref{TheConvergenceTheoremSection}).  We would like to understand the asymptotics and the convergence of $\mu_t$ as $t \to 0$. 
top-dimensional meromorphic form on $X_t$, with poles along $B_t := B|_{X_t}$. Then, $i \eta_t \wedge \overline{\eta_t}$ is a volume form on $X_t \setminus B_t$ which gives rise to a measure $\mu_t$ on $X_t \setminus B_t$.  

One way to study the convergence would be to think of the measures $\mu_t$ as being measures on $X$ with support $X_t$. However, since there is no fiber over the origin, the measures $\mu_t$ converge weakly to the zero measure on $X$ as $t \to 0$, which is not very interesting. This is where non-Archimedean geometry comes in handy.  

We restrict our attention to the case when

\begin{itemize}

\item the pair $(X,B)$ is projective and meromorphic over $\D^*$ i.e.\ $X$ is a closed subset of $\P^N \times \D^*$ for some $N \in \N$ and $X$ and $B$ are cut out by homogeneous polynomials whose coefficients are holomorphic function on $\D^*$ and meromorphic on $\D$.  
  
\item the pair $(X,B)$ has analytical singularities at $0$, i.e. there exists a proper variety $\X$ over $\D$ with $\X|_{\D^*} \simeq X$, there exists a line bundle $\L$ extending $K_{X/\D^*} + B$, and a global section $\psi$ of $\L$ which extends the generating section of $K_{X/\D^*}$ used to define $\mu_t$. Such an $\X$ is called a model of $X$. 
\end{itemize}

For $(X,B)$ satisfying the first condition above, we can can construct varieties $X_{\Ct}$ and $B_{\Ct}$ over the non-Archimedean field $\Ct$ by considering the coefficients of the polynomials cutting out $X$ and $B$ as elements of $\Ct$.

The Berkovich analytification of a variety $Y$ over the field $\Ct$, denoted $Y^\an$, is a topological space whose points are valuations on the residue fields of (scheme)  points in $Y$  that extend the $t$-adic valuation on $\Ct$ \cite{berkovich1993etale} \cite{Ber90}. By considering the Berkovich analytifications, we obtain compact Hausdorff spaces $X_{\Ct}^\an$ and $B_{\Ct}^\an$.

The main tool that we use to study the asymptotics of $\mu_t$ is a hybrid space. Various hybrid spaces, i.e.~spaces which are obtained by gluing complex analytic spaces with  non-Archimedean spaces, have been constructed in the literature. They have been used to study compactifications \cite{Oda17} and degenerations \cite{Fav16} \cite{BJ} \cite{Sch19}. Hybrid spaces were used in \cite{BJ} to answer to prove Theorem \ref{MainTheoremConvergenceLogCanonical} below for sub-klt pairs $(X,B)$. 
Following \cite{KS} \cite{Ber09} \cite{BJ}, we  construct a hybrid topological space $(X,B)^\hyb$, which as a set is a disjoint union of $X \setminus B$ and $X_{\Ct}^\an \setminus B^\an_{\Ct}$. The topology on the hybrid space is given by the logarithmic rate of convergence (See Section \ref{HybridSpaceSection}  for more details). 

We have the following convergence theorem for measures on $(X,B)^\hyb$. 
\begin{mainthm}
\label{MainTheoremConvergenceLogCanonical}
Suppose $(X,B)$ is as above. In addition, assume that the pair $(X,B)$ is sub-log-canonical. Then, there exists a measure $\mu_0$ on $X_{\Ct}^\an \setminus B_{\Ct}$ and constants $d \in \N$ and $\kappa_{\min} \in \Q$ such that the measures $\frac{\mu_t}{|t|^{2\kappa_{\min}} (2\pi\log|t|^{-1})^d}$ converge weakly to $\mu_0$, when viewed as measures on $(X,B)^\hyb$.  
\end{mainthm}

The measure $\mu_0$ is easy to describe when $(X,B)$ is log-smooth i.e. when $X$ is smooth and $B$ has snc support(See Section \ref{SectionConvergenceOfMeasure}). In this case, the support of $\mu_0$ is the locus where a certain weight function associated to $(X,B,\eta)$, constructed in \cite{MN} \cite{BM19}, is minimized. The minimizing locus of the weight function is called as the essential skeleton in the literature, and thus we have that the our measure $\mu_0$ is supported on the essential skeleton. In general, the support of $\mu_0$ is the image of a skeleton under a birational map $(X',B') \to (X,B)$, and its support is the generalization of the essential skeleton constructed by Temkin in \cite{Tem16}.  
If the pair $(X,B)$ is not sub-log-canonical, then there is no reasonable convergence in this non-Archimedean setting (See Example \ref{SubLogCanonicalNecessary} for more details). This is consistent with the observation that the essential skeleton of $(X,B,\eta)$ is empty when $(X,B)$ is not sub-log-canonical. 

As an application of Theorem \ref{MainTheoremConvergenceLogCanonical}, get a convergence result for a torus $T = (\C^*)^n$. We have  a canonical embedding $\R^n \hookrightarrow T_\Ct^{\an}$ given by sending $r \in \R^n$ to the valuation $\sum_{m \in \Z^n}a_mz^{m} \mapsto \max_{m}\{ |a_m|e^{\langle r, m \rangle} \}$. Consider the constant family $T \times \D^*$ and the associated hybrid space $(T \times \D^*) \cup T_\Ct^\an$. Then by applying Theorem \ref{MainTheoremConvergenceLogCanonical} to a smooth projective toric compactification of $T$ we get that as $t \to 0$, the Haar measure on $T \times \{t\}$ scaled by a factor of $\frac{1}{(2\pi \log |t|^{-1})^n}$ converge weakly to the Lebesgue measure on $\R^n$. See Examples \ref{EgToricVarieties} and  \ref{EgToricVariety2} for more details.

The motivation for this problem comes from \cite{BJ}, where the case for smooth $X$ and $B = 0$  \cite[Theorem A]{BJ} and the case for sub-klt pairs \cite[Theorem 8.4]{BJ} is studied. The essential difference in our scenario is that the measures $\mu_t$ are no longer finite measures when we drop the assumption that $B$ is sub-klt.

For a smooth $X$ and a smooth model $\X$ of $X$, there is an associated CW complex $\Delta(\X)$ given by the dual intersection complex of the central fiber $\X_0$. 
In \cite{BJ}, Boucksom and Jonsson construct a locally compact Hausdorff hybrid space $\X^\hyb$ over $\D$, whose fiber over $\D^*$ is $X$ and the fiber over $0$ is $\Delta(\X)$. Then, they show that the measures $\mu_t$, scaled appropriately, converge to a weighted Lebesgue measure $\mu_0$ on a subcomplex of $\Delta(\X)$. Using this, they show a convergence of the measures to a measure on $X^\an$. 

We will employ a similar approach. To prove Theorem \ref{MainTheoremConvergenceLogCanonical}, we first prove Theorem \ref{ThmA} below, which shows the convergence on certain skeletal subsets of $X^\an \setminus B^\an$. Since our measures are no longer finite, we would have to allow for the limit measures to be infinite and this would not be possible if we use Lebesgue measure on a compact simplicial complex. The solution is to allow our simplices to have unbounded faces. Now assume that $(X,B)$ is log-smooth and pick a model $\X$ such that $(\X_0 + \overline{B})_\red$ is an snc divisor. A good candidate for this is $\Delta(\X,B)$, the dual intersection complex of a pair, introduced in \cite{Tyo12} \cite{BPR13} \cite{BPR16} in the one-dimensional case and in \cite{gubler2016skeletons} \cite{BM19} for higher dimensions.

 Let $\X_0 = \sum_{i} b_i E_i$ and let $Y$ be a stratum of $\X_0 + \overline{B}$ i.e.\ a connected component of $\left(\bigcap_{i \in I} E_i\right) \cap \left(\bigcap_{j \in J} \overline{B}_j\right)$. Then, we get a face $\sigma_{Y} = \{(\underline{r},\underline{s}) \in \R_{\geq 0}^{|I| + |J|} | \sum_{i \in I}b_i r_i = 1 \}$ of $\Delta(\X,B)$. These faces are then glued together via some attaching maps to get $\Delta(\X,B)$. 

 Associated to such a model,  we construct a similar hybrid space $(\X,B)^\hyb = (X \setminus B) \cup \Delta(\X,B)$, where the topology is given by logarithmic rate of convergence. We prove the following convergence theorem on the hybrid space. Note that for Theorem \ref{MainThmSmoothLC}, we don't need to assume that $(X,B)$ is projective.

\begin{mainthm}
	\label{MainThmSmoothLC}
	\label{ConvergenceTheorem}
	\label{ThmA}
	Let $X \to \D^*$ be a holomorphic family of proper complex manifolds. Let $B$ be a snc $\Q$-divisor such that $K_{X/\D^*} + B \sim_{\Q} 0$ and  the pair $(X,B)$ is sub-log-canonical i.e.~if $B = \sum_j \beta_j B_j$ for prime divisors $B_j$, then $\beta_j \leq 1$ for all $j$. 
        Let $\X$ be a smooth proper model of $X$ such that $\X_0 + B$ is snc, let $\L$ extend $K_{X/\D^*} + B$ and let $\psi \in H^0(\X,m\L)$ be a generating section for sufficiently divisible $m$ and let $\mu_t$ be the measure induced on $X_t$ by $\psi$. Then, there exists a `Lebesgue-type' measure $\mu_0$ supported on a subcomplex $\Delta(\L)$ of $\Delta(\X,B)$ and explicit constants $d \in \N$, $\kappa_\min \in \Q$ such that 
	$$ \frac{\mu_t}{|t|^{2\kappa_{\min}} (2\pi\log|t|^{-1})^d} \to \mu_0$$
	converges weakly as measures on $(\X,B)^\hyb$. 
	
\end{mainthm}

For precise details, see Section \ref{SectionConvergenceOfMeasure}.

We can view $\Delta(\X)$ and $\Delta(\X,B)$ as subsets of the Berkovich analytification, $X_{\Ct}^\an$. Moreover, $\Delta(\X)$ is a strong deformation retract of $X_{\Ct}^\an$ and its image is denoted as $\text{Sk}(\X)$ in the literature.

Theorem \ref{MainTheoremConvergenceLogCanonical} follows from Theorem \ref{ThmA} by using the following trick which was used in \cite{BJ}. The collection of $\Delta(\X)$ for all smooth proper snc models $\X$ is a directed system and $X_{\Ct}^\an \simeq \varprojlim_{\X}\Delta(\X)$ (See \cite[Theorem 10]{KS}, \cite[Corollary 3.2]{BFJ16}). 
We prove a similar result (see Theorem \ref{LimitSkeletonIsNonArchimedean})  that
$$X_{\Ct}^\an \setminus B_{\Ct}^\an \simeq \varprojlim_{\X}\Delta(\X,B).$$ The topology on  $(X,B)^\hyb$ is in fact given by $(X,B)^\hyb = \varprojlim_{\X}(\X,B)^\hyb$ , which immediately proves Theorem \ref{MainTheoremConvergenceLogCanonical} for smooth $X$. 

For a general $(X,B)$, by taking a  log resolution $(X',B') \to (X,B)$, and using Theorem \ref{ThmA} for $(X',B')$, we are able to prove Theorem \ref{MainTheoremConvergenceLogCanonical}. 

It would interesting to see the application of Theorem \ref{MainTheoremConvergenceLogCanonical} to various examples of log-Calabi-Yau varieties that are available in the literature \cite{Man19} to see if it yields any interesting results.

In \cite{JN19}, Jonsson and Nicaise prove a $p$-adic version of \cite{BJ}, where they consider the measure induced by a pluricanonical form $\eta$ on a smooth proper variety $X$ over a local field $K$. They show that the measures induced by $\{ \eta \otimes_K K' \}_{K'}$ for all finite tame extensions $K'$ of $K$ converge to a Lebesgue type measure on the Berkovich analytification.   
The measures considered in \cite{JN19} are finite and it would be interesting to see whether it would be possible generalize Theorem \ref{MainTheoremConvergenceLogCanonical} to this $p$-adic setting to extend the result to a family of infinite measures as well. 

\subsection*{Structure of the paper} The paper is structured as follows. In Section \ref{SectionDualSimplices}, we recall the construction of the dual complex $\Delta(\X,B)$ associated to an snc model $\X$ of a log smooth pair $(X,B)$ and in Section \ref{HybridSpaceSection},  we recall the construction of the the hybrid space $(\X,B)^\hyb$, associated to a model $\X$. 
In Section \ref{SectionConvergenceOfMeasure}, we prove Theorem \ref{ThmA}. 
In Section \ref{SectionNonArchimedean}, we construct the space $(X,B)^\hyb$, realize it and its the central fiber as a non-Archimedean space and prove Theorem \ref{MainTheoremConvergenceLogCanonical}.

\subsection*{Acknowledgments} I would like to thank my advisor, Mattias Jonsson, for suggesting this problem, and also for his support and  guidance. This work was supported by the NSF grant DMS-1600011.

\section{The dual simplicial complex associated to an snc model}
\label{SectionDualSimplices}
In this section, we recall the notion of a model and construct the dual intersection complex associated to an snc model of a log smooth pair $(X,B)$. 
Let $X$ be a holomorphic flat family of compact manifolds parametrized by $\D^*$ i.e.\ $X$ is a smooth complex manifold with a proper flat map $X \to \D^*$. Let $B$ be an snc $\Q$-divisor in $X$. Write $B = \sum_j \beta_j B_j$, where $\beta_j \in \Q$ and $B_j$ are prime divisors. In this section, we don't need to assume that $(X,B)$ is projective. 

\subsection{Models of $(X,B)$}
A model of a pair $(X,B)$ is a complex analytic space $\X$ flat over $\D$ such that  we have a specified isomorphism $\X|_{\D^*} \simeq X$ as spaces over $\D^*$. We say that a model $\X$ is snc if $\X$ is smooth and $(\X_0 + \overline{B})_\red$ is an snc divisor in $\X$. We say that $\X$ is proper if $\X$ is proper over $\D$. 
 By abuse of notation, we will also denote the closure of $B$ in $\X$ by $B$ as well. Let $\X_0 = \sum_i b_i E_i$ denote the central fiber, where $E_i$ are irreducible components of the central fiber and denote $D = \X_0 + B$.

 By Hironaka's resolution of singularities, given a proper model $\X$ of $(X,B)$, we can always produce a proper snc model $\X'$ of $(X,B)$ such that we have a proper map $\X' \to \X$ which commutes with the projection to $\D$.

\subsection{The dual complex}
To an snc model $\X$ of $(X,B)$, we can associate a CW complex (with possibly open faces) $\Delta(\X,B)$, called the dual complex, as follows. A non-empty connected component $Y$ of $E_{i_0} \cap \dots \cap E_{i_p} \cap B_{j_1} \cap \dots \cap B_{j_q}$ is called a stratum. Associated to the stratum $Y$, we have a face $$\sigma_Y = \{(x_0,\dots,x_p,y_1\dots,y_q) \in \R_{\geq 0}^{p+1} \times \R_{\geq 0}^q | \sum_i b_i x_i = 1\} \subset \R_{\geq 0}^{p+1} \times \R_{\geq 0}^q.$$

If $Y'$ is a stratum that contains $Y$, then we have attaching maps $\sigma_{Y'} \hookrightarrow \sigma_Y$ given by sending the extra coordinates to $0$. For example, if $Y' = E_{i_0} \cap \dots \cap E_{i_p'} \cap B_{j_1} \cap \dots \cap B_{j_q'}$ for some $p' \leq p$ and $q' \leq q$, then the map $\sigma_{Y'} \hookrightarrow \sigma_{Y}$ is given by $$(x_0,\dots,x_{p'},y_1,\dots,y_{q'}) \mapsto (x_0,\dots,x_{p'},0,\dots,0,y_1,\dots,y_{q'},0,\dots,0).$$ 

The union of all such faces corresponding to all the strata in $\X_0$ for all $p, q \geq 0$ along with the attaching maps, give rise to the CW complex $\Delta(\X,B)$. For example, if $\dim(X_t) = 1$, then $\Delta(\X,B)$ is the dual graph complex of $D$ with the vertices corresponding to $B$ removed. The dual complex of a pair was introduced in \cite{gubler2016skeletons} \cite{BM19}. 

The complex $\Delta(\X) := \Delta(\X,0)$, used in \cite{BJ}, is just the subcomplex of $\Delta(\X,B)$ consisting of all the bounded faces.

\begin{eg}[The dual complex associated to $\P^1 \times \D$]
	\label{P1DualComplex}
	Let $X = \P^1 \times \D^*$, with projection to $\D^*$, and $B = \{0\} \times \D^* + \{\infty\} \times \D^*$.  Consider the model $\X = \P^1 \times \D$. Then, the dual complex $\Delta(\X,B)$ is homeomorphic to $\R$, with $0$ being the vertex $\sigma_{\P^1 \times \{0\}}$, the positive axis being identified with $\sigma_{(0,0)}$ and the negative axis with $\sigma_{(\infty,0)}$. See Figure \ref{Figure1}. 
\end{eg}

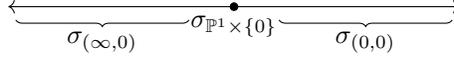
\begin{figure}
	\centering
	\begin{tikzpicture}
	\draw [<->] (-3,0)-- (0,0) -- (3,0);
	\draw [fill] (0,0) circle [radius = 0.05];
	\node [below] at (0,0) {$\sigma_{\P^1 \times \{0\}}$};
	\draw[decoration={brace,mirror,raise=5pt},decorate] (0.6,0) -- node[below=6pt] {$\sigma_{(0,0)}$} (2.9,0);
	\draw[decoration={brace,mirror,raise=5pt},decorate] (-2.9,0) -- node[below=6pt] {$\sigma_{(\infty,0)}$} (-0.6,0);
	\end{tikzpicture}
	\caption{The dual complex $\Delta(\X,B)$ for $\X = \P^1 \times \D$ and $B = \{0, \} \times \D + \{\infty\} \times \D $}
		\label{Figure1}
\end{figure}
\subsection{Integral piecewise affine structure on the dual intersection complex}
We briefly discuss some results related to the natural integral piecewise affine structure on $\Delta(\X)$. The reader can take a look at \cite{Ber99}, \cite{Ber04} and \cite[Section 1.3]{BJ} for more details. Given a polytope $\sigma = \{(x_0,\dots,x_p) | \sum_{i=0}^pb_i x_i = 1 \} \times \R_{\geq 0}^q$, consider $M_{\sigma}$, the collection of affine linear functions with integral coefficients on $\R_{\geq 0}^{p+1+q}$ restricted to $\sigma$ (and two such functions are identified if they are equal on $\sigma$). Denote $b_\sigma := \gcd(b_0,\dots,b_p)$.

Let $\textbf{1}_{\sigma} \in M_\sigma$ denote the constant function $1$ on $\sigma$.  The evaluation map $\sigma \to (M_\sigma)_\R^{\vee}$ realizes $\sigma$ as a (possibly unbounded) polytope of codimension one in $(M_\sigma)_\R^{\vee}$ contained in the affine plane $\{\nu | \nu(\textbf{1}_\sigma) = 1 \}$. So, the tangent space of $\sigma$ in $(M_\sigma)_\R^{\vee}$ can be realized as $(\overrightarrow{M_\sigma})_{\R}^\vee$, where 
$$ \overrightarrow{M_\sigma} = M_\sigma/{\Q \textbf{1}_{\sigma} \cap M_{\sigma}} $$

The Lebesgue measure on $(\overrightarrow{M_\sigma})_{\R}^\vee$, normalized with respect to the lattice $\mathrm{Hom}_{\mathbb{Z}}(\overrightarrow{M_\sigma}), \mathbb{Z})$ gives rise to a measure on $\sigma$. This is called the normalized Lebesgue measure $\lambda_\sigma$ of $\sigma$. The following remark, stated with a typo in \cite[Remark 1.3]{BJ}, gives an explicit description of the normalized Lebesgue measure, which will be useful for computations. We provide a quick proof here for the convenience of the reader.  

\begin{prop}[{\cite[Remark 1.3]{BJ}}]
	\label{NormalizedLebesgueMeasureAnalysis}
	Let $b_0,\dots,b_p \in \N_{+}$ and let $$\sigma = \{(x_0,\dots,x_p,y_1,\dots,y_q) \in \R^{p+q+1}_{\geq 0} | \sum_{i=0}^pb_i x_i = 1 \}$$ be a polytope. Then, we have a homeomorphism $$\sigma \to \{(x_1,\dots,x_p,y_1,\dots,y_q) \in \R^{p+q} | \sum_{i=1}^pb_i x_i \leq 1 \},$$ where we can recover $x_0$ by $x_0 = b_{0}^{-1}(1 - \sum_{i=1}^p b_i x_i)$. Under this homeomorphism, the normalized Lebesgue measure is given by  $$\lambda_\sigma = b_{\sigma}b_0^{-1}| dx_1 \wedge \dots \wedge dx_p \wedge dy_1 \wedge \dots \wedge dy_q |$$
\end{prop}
\begin{proof}
	Note that $1_{\sigma},X_1,\dots,X_p,Y_1,\dots,Y_q$ is a $\R$-basis for $(M_{\sigma})_{\R}$ as an abelian group, where $X_i$ and $Y_j$ denote projection to the $x_i$ and $y_j$ coordinates. Let $1_\sigma^*, X_1^*,\dots,Y_q^*$ denote the dual basis. Then, $X_1^*,\dots,Y_q^*$ is a $\R$-basis for the $(\overrightarrow{M_\sigma})_{\R}^\vee$ and $\Hom_{\mathbb{Z}}(\overrightarrow{M_\sigma}, \mathbb{Z})$ is a sub lattice of $\Lambda = \mathbb{Z}X_1^* + \dots + \mathbb{Z}Y_q^*$.
	
	Note that we can view $\Hom_{\mathbb{Z}}(\overrightarrow{M_\sigma}, \mathbb{Z})$ as the kernel of the map $\phi : \Lambda \to \Z/b_0\Z$ given by $\alpha_1 X_1^* + \dots + \alpha_p X_p^* + \gamma_1 Y_1^* + \dots + \gamma_q Y_q^* \to b_1 \alpha_1 + \dots + b_p \alpha_p + b_0 \Z$. The image of $\phi$ is generated by $b_\sigma$ and the size of the image is $\frac{b_0}{b_\sigma}$. Thus, the index of $\Hom_{\mathbb{Z}}(\overrightarrow{M_\sigma}, \mathbb{Z})$ in $\Lambda$ is $\frac{b_0}{b_\sigma}$, and thus $b_\sigma b_0^{-1}|dx_1\wedge \dots \wedge dx_p \wedge dy_1 \wedge \dots \wedge dy_q|$ is the normalized Lebesgue measure on $\sigma$.	
\end{proof}

\section{The hybrid space associated to a dual complex}
\label{HybridSpaceSection}
In this section, we construct a hybrid space $(\X,B)^\hyb$, associated to an snc model $\X$ of a log-smooth pair $(X,B)$; this is a topological space over $\D$ such that the fiber over $\D^*$ is isomorphic to $X$ and the central fiber is isomorphic to $\Delta(\X,B)$. This construction exactly follows \cite[Section 2.2]{BJ}, where the construction for $B = 0$ was done. 

\subsection{Local Log function}
To construct the hybrid space, we will first construct a $\Log$ function on this space as done in \cite{BJ} and glue $X$ and $\Delta(\X,B)$ using this $\Log$ function. To do this, we first construct a local version of the $\Log$ function. For an open set $U \subset \X$ and for local coordinates $(z,w,y)$ on $U$ where $z = (z_0,\dots,z_p)$, $w = (w_1,\dots,w_q)$ and $y = (y_1,\dots,y_r)$, we say that $(U,(z,w,y))$ is adapted to a stratum $Y =_{\text{locally}} E_{0} \cap \dots \cap E_{p} \cap B_{1} \cap \dots \cap B_{q}$ if 
\begin{itemize}
	\item The only irreducible components of $D$ intersecting $U$ are $E_{0},\dots,E_{p}$, and $B_{1},\dots,B_{q}$
    \item $U \cap (\ E_{0} \cap \dots \cap E_{p} \cap B_{1} \cap \dots \cap B_{q}) = U \cap Y$.
    \item We have $|z_i|,|w_j|,|y_k| < 1$ on $U$ and $E_i \cap U = \{z_i = 0\}$ and $B_j \cap U = \{w_j = 0\}$.
\end{itemize}
In such a case, we can define $\Log_U : U \setminus D \to \sigma_Y$. Let $f_U := z_0^{b_0}\dots z_p^{b_p}$. Then, $t = u\cdot f_U$ for some invertible holomorphic function $u$  on $U$. Define $$\Log_U(z,w,y) = \left(\frac{\log|z_0|}{\log|f_U|},\dots,\frac{\log|z_p|}{\log|f_U|},\frac{\log|w_1|}{\log|f_U|},\dots,\frac{\log|w_q|}{\log|f_U|}\right).$$

\begin{remark}[{\cite[Prop 2.1]{BJ}}]
If $(U,(z,w,y))$ and $(U',(z',w',y'))$ are adapted to a stratum $Y$, then 
$$ \Log_U - \Log_{U'} = O\left(\frac{1}{\log|t|^{-1}}\right)     $$ as $t \to 0$ uniformly on compact subsets of $U \cap U'$.
\end{remark}

\subsection{Constructing the global Log function}
Here, we globalize the log construction by patching up the local log functions and to do so, we will have to find a `nice' open covering of $D$. The following construction, as well as Proposition \ref{GlobalLogWellDefined}  is similar to  \cite[Proposition 2.1]{BJ}, but we provide some more details.

Following \cite[Theorem 5.7]{Cle77}, we can find tubular neighborhoods $U_{I,J}$ of $D_{I,J} := E_{I} \cap B_{J}$ and a smooth projection $\pi_{I,J} : U_{I,J} \to D_{I,J}$ satisfying $U_{I,J} \cap U_{I',J'} = U_{I \cup I', J \cup J'}$. In particular, if $U_{I,J}$ and $U_{I',J'}$ intersect, then $D_{I,J} \cap D_{I',J'} \neq \emptyset$. Also, note that $U_{I,J}$ has as many connected components as $D_{I,J}$ and each connected component $U_Y$ of $U_{I,J}$ corresponds to a stratum $Y \subset E_I \cap B_J$. 

Pick $x \in \X_0$. Suppose $x$ lies in a stratum $Y$. Around $x$, pick an open neighborhood $U_x$ that is adapted to $Y$ and lies in $U_Y$. The union of all such $U_x$ for $x \in \X_0$ covers $\X_0$. Since $\X_0$ is compact, we only need finitely many of these. Call these open sets $U_1, \dots, U_l$ and let their corresponding strata be $Y_1$, \dots, $Y_l$ respectively. Let $\chi_1,\dots,\chi_l$ be a partition of unity with respect to $U_1,\dots,U_l$ and let $V = \bigcup_{\lambda = 1}^l U_\lambda$. Then, $V$ is a neighborhood of $\X_0$. 

\begin{prop}
\label{GlobalLogWellDefined}
The function $\Log_V : V \setminus D \to \Delta(\X,B)$ given by $\Log_V = \sum_{\lambda = 1}^l \chi_\lambda \Log_{U_\lambda}$ is well defined.
\end{prop}
\begin{proof}
Pick a point $x \in V$. After a possible re-indexing, suppose $x \in (U_1 \cap \dots \cap U_a) \setminus (U_{a+1} \cup \dots \cup U_l)$. Then, $\Log_V(x) = \chi_1(x) \Log_{U_1}(x) + \dots + \chi_a(x) \Log_{U_a}(x)$. For this to make sense, it is enough to find a face $\sigma'$ of $\Delta(\X,B)$ such that $\sigma_{Y_1},\dots,\sigma_{Y_a} \subset \sigma'$. Note that $U_1 \cap \dots \cap U_a \subset U_{Y_1} \cap \dots \cap U_{Y_l}$. Each connected component of $\bigcap_{\lambda = 1}^l U_{Y_\lambda}$ corresponds to a stratum of $\bigcap_{\lambda = 1}^lY_{\lambda}$. Let $Y'$ be the stratum corresponding to the connected component of $\bigcap_{\lambda = 1}^l U_{Y_\lambda}$ containing $x$. Then, $\sigma' := \sigma_{Y'}$ contains $\sigma_{Y_\lambda}$ for all $\lambda = 1,\dots,l$.
\end{proof}

\begin{prop}
\label{LogDiffersBy}
Let $U$ be an open set adapted to a stratum $Y$. Then, $\Log_V - \Log_{U} = O(\frac{1}{\log|t|^{-1}})$ locally uniformly as $t \to 0$, where the equality is interpreted as being true in some faces of $\Delta(\X,B)$ containing $\sigma_Y$, for all $\lambda = 1,\dots,l$. 
\end{prop}
\begin{proof}

We may replace $U$ by $U \cap U_Y$ and assume that $U \subset U_Y$.
Suppose $x \in (U_1 \cap \dots \cap U_a) \setminus (U_{a+1} \cup \dots \cup U_l)$. Then, from the previous proof, we know that there exists a stratum $Y'$ such that $x \in U_{Y'}$. Since $x \in U_Y \cap U_{Y'}$, which tells us that $Y \cap Y' \neq \emptyset$. Let $Z$ be the stratum corresponding to the connected component of $U_Y \cap U_{Y'}$ containing $x$. Then, $\sigma_{Y}, \sigma_{Y'} \subset \sigma_Z$.

Suppose $x \in E_i$, $z_i = 0$ defines $E_i$ in $U$, and $z_i' = 0$ defines $E_i$ in $U_1$. Then, $\frac{\log|z_i|}{\log|f_U|} - \frac{\log|z_i'|}{\log|f_{U_1}|} = O(\frac{1}{\log|t|^{-1}})$ in a neighborhood of $x$.

Suppose $x \notin E_i$ and $z_i' = 0$ defines $E_i$ in $U_1$. Then, $\frac{\log|z_i'|}{\log|t|} = O(\frac{1}{\log|t|^{-1}})$ in a neighborhood of $x$. 

Using a similar argument for $B_j$'s as well gives us that $\Log_U - \Log_{U_1} = O(\frac{1}{\log|t|^{-1}})$ in a neighborhood of $x$. Repeating the argument for all $U_i$ for $i = 1,\dots,a$, we get that $\Log_U - \Log_V = O(\frac{1}{\log|t|})$ in a neighborhood of $x$.

\end{proof}

\subsection{The hybrid space}
The hybrid space of an snc model $\X$ of $(X,B)$, as a set, is defined as $(\X,B)^\hyb := (X \setminus B) \cup \Delta(\X,B)$. The topology on the hybrid space is defined by 
\begin{itemize}
\item $X\setminus B \hookrightarrow (\X,B)^\hyb$ is an open immersion.
\item The projection map $\pi : (\X,B)^\hyb \to \D$ given by extending the projection $X \setminus B \to \D^*$ and sending $\Delta(\X,B)$ to the origin is continuous. 

	\item $\Log_{V}^\hyb : (V \setminus (\X_0 + B)) \cup \Delta(\X,B) \to \Delta(\X,B)$ defined by $\Log_V$ on $V \setminus (\X_0 + B)$ and identity on $\Delta(\X,B)$ is continuous. 
        \end{itemize}
        
Note that the hybrid space does not contain $B$. It follows from Proposition \ref{LogDiffersBy} that the topology of the hybrid space does not depend on the global log function we pick. Also note that the fiber of $\pi  : (X,B)^\hyb \to \D$ over $t \in \D^*$ is $X_t \setminus B_t$. 

\begin{eg}[Hybrid space of $\P^1 \times \D$]
	\label{HybridSpaceOfP1TimesD}
	The hybrid space $(\X,B)^\hyb$ for Example \ref{P1DualComplex} is given by $\C^* \times \D$ with the identification $(re^{i\theta_1},0) \sim (re^{i\theta_2},0)$ for all $r \in \R, \theta_i \in [0,2\pi]$. Over any line segment in $\D$ with one end point $0$, $(\X,B)^\hyb$ is given by a solid cylinder. See Figure \ref{Figure2}.  
	\begin{figure}
		\centering
		\begin{tikzpicture}
		\draw[fill] (0,0) circle [radius = 0.05]; 
		\draw [<->] (-3.5,0) -- (0,0) -- (3.5,0);
		\draw[] (-2.5,1) -- (0,1) -- (2.5,1);
		\draw[]  (-2.5,-1) -- (0,-1) -- (2.5,-1);
		\draw (-2.5,1) arc (90:270:0.3 and 1);
		\draw (2.5,1) arc (90:270:0.3 and 1);
		\draw[dotted, very thick] (-2.5,1) arc (90:-90:0.3 and 1);
		\draw[dotted, very thick] (2.5,1) arc (90:-90:0.3 and 1);
		\node [right, below] at (0,-1.5) {$(\P^1 \times \D,\{0\} \times \D + \{\infty\} \times \D )^\hyb$};
		
		\draw[->] (4.5,0) -- (5.25,0); 
		
		\draw[] (6,1.5) -- (6,0) -- (6,-1.5); 
		\draw[fill] (6,0) circle [radius = 0.05];
		\draw[fill] (6,1) circle [radius = 0.05];
		\node [right] at (6,0) {0};
		\node [right] at (6,1) {$t$};
		\node [right, below] at (6.1,-1.5) {$\D$};
		\end{tikzpicture}
		\caption{The hybrid space $(\P^1 \times \D,\{0\} \times \D + \{\infty\} \times \D )^\hyb$ with the projection to $\D$}
		\label{Figure2}
	\end{figure}
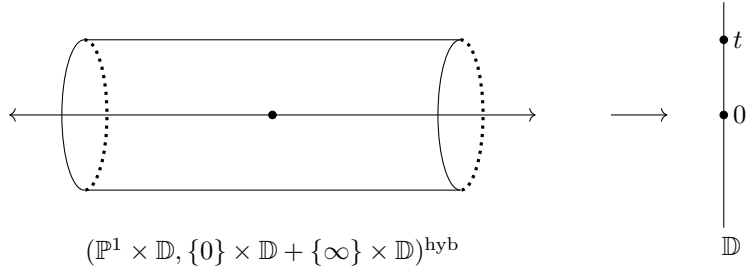
\end{eg}

The hybrid space $\X^\hyb$, constructed in \cite{BJ} in the case $B = 0$ is compact over a closed neighborhood of the origin. But the hybrid space $(\X,B)^\hyb$ that we construct is not always compact over a neighborhood of the origin, as can be seen from Example \ref{HybridSpaceOfP1TimesD}. However, the following proposition tell us that it is not too bad. In particular, it implies that the hybrid space is locally compact. 

\begin{prop}
	The map $\Log_{V^\hyb} : ((V\setminus D) \cup \Delta(\X,B)) \to \Delta(\X,B)$ is proper near the central fiber, in the sense that for a compact set $K \subset \Delta(\X,B)$, $\Log_{V^\hyb}^{-1}(K) \cap \pi^{-1}(\overline{\frac{1}{2}\D})$ is a compact subset of $(\X,B)^\hyb$. 
\end{prop}
\begin{proof}

By rescaling the coordinate $t$, we may without loss of generality assume that $V = \X$.  We need to show that $L = \Log_{V^\hyb}^{-1}(K) \cap \pi^{-1}(\overline{\frac{1}{2}\D})$ is compact. Let $\bigcup_{i \in I}U_i$ be an open cover of $L$. Since $K \subset L$ is compact, there exists a finite subset $I' \subset I$ such that $K \subset \bigcup_{i \in I'} U_i$. For a point $P \in \Delta(\X,B) \subset (\X,B)^\hyb$, the sets of the form $\Log^{-1}_{V^\hyb}(W) \cap \pi^{-1}(\frac{1}{N}\D)$, where $W \subset \Delta(\X,B)$ is an open neighborhood of $P$ in $\Delta(\X,B)$ and $N \in \mathbb{N}$, form basic open neighborhoods of $P$ in $(\X,B)^\hyb$. Since $K$ is compact, there exists $0 < r \ll 1$ such that $\Log^{-1}_{V^\hyb}(K) \cap \pi^{-1}(r\D) \subset \bigcup_{i\in I'}U_i$.
    
    Since $L \cap \pi^{-1}\{ \frac{r}{2}\leq |t| \leq \frac{1}{2} \}$ is a closed subset of the compact set $\pi^{-1}\{ \frac{r}{2}\leq t \leq \frac{1}{2} \}$, it is compact. Thus, we can find a finite subset $J' \subset I$ such that $L \cap \pi^{-1}\{ \frac{r}{2}\leq |t| \leq \frac{1}{2} \} \subset \bigcup_{i \in J'}U_i$. Thus, $L \subset \cup_{i \in I' \cup J'}U_i$.
\end{proof}

\section{Convergence of measure}
\label{SectionConvergenceOfMeasure}
In this section, we prove Theorem \ref{ThmA} by imitating the proof of \cite[Theorem A]{BJ}.  The proof idea is the same, except for some new calculations. Since $(\X,B)^\hyb$ is not compact, we can no longer use Stone-Weierstrass as done in \cite{BJ}. Instead, we use Lemma \ref{CompactSetFiniteVolume}. Let $(X,B)$ be as in the previous section. Further assume that $K_X + B \sim_{\Q} 0$ and $(X,B)$ is sub-log canonical i.e. $\beta_j \leq 1$ for all $j$. Fix a proper snc model $\X$ of the pair $(X,B)$. Note that  we still don't need to assume that $X$ is projective in this section as well.

\subsection{The subcomplex $\Delta(\L)$ of $\Delta(\X,B)$}
Suppose $\L$ is a $\Q$-line bundle on $\X$ that extends $K_{X/\D^*} + B$. Then, $\L$ differs from $K^\log_{\X/\D} + B$ only by vertical divisors, where $K_{\X/\D}^\log = K_\X - \X_0 + (\X_0)_\red$. Thus, we can write  $\L = K^\log_{\X/\D} - \sum_i a_i E_i + \sum_j \beta_j B_j$ for some $a_i \in \Q$ (See \cite[Section 3.1]{BJ}).  
 Let $\kappa_i = \frac{a_i}{b_i}$ and $\kappa_{\min} = \min_i \kappa_i$.

Since $X$ is smooth, the condition that $(X,B)$ is sub-log-canonical is equivalent to saying that $\beta_j \leq 1$ for all $j$. Define the subcomplex $\Delta(\L) \subset \Delta(\X,B)$ as follows. If $Y \subset E_I \cap B_J$, then $\sigma_Y \in \Delta(\L)$ if $\kappa_i = \kappa_\min$ for all $i \in I$ and if $\beta_j = 1$ for all $j \in J$. In the case when $\dim(X_t) = 1$, this just means that we pick the subgraph generated by vertices corresponding to irreducible components with minimal $\kappa$-value and the rays corresponding to intersections $E_i \cap B_j$ with $\kappa_i = \min_k \kappa_k$ and $\beta_j = 1$. 

Define $b_{\sigma_Y} = \gcd(b_i)_{i \in I}$ and let $\lambda_{\sigma_Y}$ be the normalized Lebesgue measure on $\sigma_Y$. Define $d := \dim(\Delta(\L))$

\subsection{The residual measure}
Given a section $\psi \in H^0(\X,m\L)$ and a closed subvariety $Y \subset \X_0$, we can get a section $\Res_{Y}(\psi) \in H^0(Y,m(\L - \sum_{j|Y \subset B_j} B_j - \sum_{i|Y \subset E_i}E_i)|_{Y})$. If $z_0,\dots,z_p,w_1,\dots,w_q = 0$ define $Y$ locally, and \linebreak
$\psi = f\left( \frac{dz_0}{z_0} \wedge \dots \wedge \frac{dz_p}{z_p} \wedge \frac{dw_1}{w_1} \wedge \dots \frac{dw_q}{w_q} \wedge \phi \right)^{\otimes m}$. Then, $\Res_{Y}(\psi) = f \cdot \phi|^{\otimes m}_{Y}$. Then, $|\Res_{Y}(\psi)|^{2/m}$ gives rise to a positive measure on $Y \setminus (\cup_{i|Y \not{\subset} B_i}B_i|_{Y})$. We denote $\int_{Y \setminus (\cup_{i|Y \not{\subset} B_i}B_i|_{Y})} |\Res_{Y}(\psi)|^{2/m} $ by just $\int_{Y} |\Res_{Y}(\psi)|^{2/m}$. 

\subsection{The Convergence Theorem}
\label{TheConvergenceTheoremSection}

Let $n + 1$ denote the dimension of $X$ i.e. each of the fibers $X_t$ for $t \neq 0$ has dimension $n$.  Let $\eta \in H^0(X, m(K_{X/\D^*} + B))$ be a generator and suppose there exists a section $\psi \in H^0(\X,m\L)$ that extends $\eta$. Let $\psi_t$ denote the restriction $\psi|_{X_t}$ for $t \neq 0$. If $\psi_t = \alpha \cdot (dx_1\wedge \dots \wedge dx_N)^{\otimes m}$ on a local chart, then $i^{n^2}(\psi_t \wedge \overline{\psi_t})^{1/m}$ given locally by $$i^{n)^2}(\psi_t \wedge \overline{\psi_t})^{1/m} = |\alpha|^{2/m} dx_1\wedge d\overline{x}_1  \wedge \dots dx_N \wedge d\overline{x}_N$$ is a well-defined positive smooth volume form on $X_t \setminus B_t$.    

Define a measure $$\mu_t = \frac{i^{n^2}}{|t|^{2\kappa_{\min}}(2\pi \log|t|^{-1})^d}(\psi_t \wedge \overline{\psi_t})^{1/m}$$

on $X_t \setminus B_t$, and a measure
$$\mu_0 := \sum_{\sigma \subset_{\text{face}} \Delta(\L), \dim(\sigma) = d} \left(\int_{Y_\sigma} |\Res_{Y_\sigma}(\psi)|^{2/m} \right) b_{\sigma}^{-1} \lambda_\sigma$$ on $\Delta(\X,B)$.

\begin{eg}
  \label{SubLogCanonicalNecessary}
  This example illustrates the importance of the sub-log-canonical assumption. For simplicity, assume that $X$ has relative dimension 1. Let $E_0$ be an irreducible component of $\X_0$ and let $B_0$ be an irreducible component of $B$ occurring with multiplicity $\beta_0 > 1$. Let $\sigma \simeq \R_{\geq 0}$ be the face corresponding to $E_0 \cap B_0$. Let $z$ and $w$ denote the functions that define $E_0$ and $B_0$ in an open neighborhood $U$ of $E_0 \cap B_0$ such that $|z|,|w| < 1 $ on $U$. We may assume that $t = z^{b_0}$

  We have $\Log_U : (U \setminus (E_0 + B_0)) \to \R_{\geq 0}$ given by $(z,w) \mapsto \frac{\log|w|}{\log|t|}$. Suppose we had that $(\Log_U)_*(\alpha(t) \mu_t)$ weakly converged to a measure $\mu_0$ on $\R_{\geq 0}$ for some positive scaling function $\alpha(t)$. By scaling by a suitable power of $|t|$, we may assume that $\mu_t = i |w|^{-2\beta_0} dw \wedge d\overline{w}$. Pick a compactly supported continuous function $f$ on $\R_{\geq 0}$.
  Then,
  $$\int_{U_t} (f \circ \Log_U)d\mu_t = \int_{U_t} f\left(\frac{\log|w|}{\log|t|}\right) i |w|^{-2\beta_0}  dw \wedge d\overline{w}  .$$ Making a change of variable $w = |t|^{u}e^{i\theta}$, we get
$$ \int_{U_t} (f \circ \Log_U)d\mu_t = \frac{2\pi}{(\log|t|^{-1})} \int_{0}^\infty f(u) |t|^{-2(\beta_0 - 1)u}du$$ 

If we pick a function $f$ that is close to the indicator function of $[0,N]$, then $\alpha(t) \int_{U_t} (f \circ \Log_U)d\mu_t = O(\frac{\alpha(t)}{\log|t|^{-1}} \frac{|t|^{-2(\beta_0 - 1)N}}{\log|t|^{-1}})$ as $t \to 0$. If we require that this expression converge for all values of $N$ as $t \to 0$, then it is easy to see that this is only possible if $\mu_0$ is the zero measure and $\frac{1}{\alpha(t)}$ is growing super-polynomially as $t \to 0$. Thus, we see that the convergence in this hybrid space setting is not very interesting if don't assume that $(X,B)$ is sub-log-canonical. 
\end{eg}

To prove Theorem \ref{ConvergenceTheorem}, we first prove a local version for functions that are pulled-back from a face $\sigma_{Y}$ via a local $\Log$ map. 
\begin{lem}
\label{MainLemmaConvergenceThorem}
Let $U$ be a coordinate chart adapted to a stratum $Y$ of $\X_0$. Let $f$ be a compactly-supported continuous real-valued function on $\sigma_U$. Let $\chi \in C_c(U)$ and if a maximal face of $\Delta(\L)$ is contained in $\sigma_Y$, let $\sigma_{Y'}$ denote this (unique) maximal face and let $Y'$ be the stratum associated to $\sigma_{Y'}$.

 If a maximal face of $\Delta(\L)$ is contained in $\sigma$, then $$\int_{U \cap X_t} (f\circ \Log_U) \chi d\mu_t \to \left( \int_{Y'}\chi|\Res_{Y'}(\psi)|^{2/m}\right)\int_{\sigma_{Y'}} f b_{\sigma_{Y'}}^{-1} \lambda_{\sigma_{Y'}}$$ as $t \to 0$. 
 If $\sigma_Y$ does not contain a maximal face of $\Delta(\L)$, then the above limit is $0$.
\end{lem}
\begin{proof}
By replacing $\L$ by $\L - \kappa_\min\X_0$ and $\psi$ by $t^{\kappa_\min}\psi$, we may assume that $\kappa_\min = 0$. Suppose $Y = E_0 \cap \dots \cap E_p \cap B_1 \cap \dots \cap B_q $ locally. The proof for the case $q = 0$ can be found in \cite[Lemma 3.5]{BJ}, and the calculations in this proof are not very different. The only new estimate we need to make is Equation \eqref{eqnNew}. 
Let $(z,w,y)$ be coordinates on $U$ such that $E_i = \{z_i = 0\}$ and $B_j = \{w_j = 0\}$ on $U$. To simplify notation, denote $\underline{z}^{\underline{a}} := z_0^{a_0}\dots z_p^{a_p}$ and $\underline{w}^{\underline{\beta}}:= w_1^{\beta_1}\dots w_{q}^{\beta_q}$. 
Then, we can write $\psi^{1/m}$ locally in $U$ as 
 $$\psi^{1/m} = \underline{z}^{\underline{a}}\cdot \underline{w}^{-\underline{\beta}} \cdot \frac{dz_0}{z_0} \wedge \dots \wedge \frac{dz_p}{z_p} \wedge dw_1 \wedge \dots \wedge dw_q \wedge d\underline{y}$$

We know that $\psi_t^{1/m} \wedge \frac{dt}{t} = \psi^{1/m}$. Using $\frac{dt}{t} = \sum_{i=0}^p b_i\frac{dz_i}{z_i}$ we can write $\psi_t^{1/m}$ on $U \cap X_t$ as  

\begin{equation}
\label{SimplifiedPsiT}
	\psi_t^{1/m} = \frac{ \underline{z}^{\underline{a}}\cdot \underline{w}^{-\underline{\beta}+1}}{b_0}  \frac{dz_1}{z_1} \wedge \dots \wedge \frac{dz_p}{z_p} \wedge \frac{dw_1}{w_1} \dots \wedge \frac{dw_q}{w_q} \wedge d\underline{y} \Big|_{X_t}.
\end{equation}

Denote by $\Log_t : ((X_t \cap U) \setminus D) \to \sigma_Y \times Y$ the map given by $\Log_t(z,w,y) = (\Log_U(z,w),y)$. 

Similar to the analysis done in \cite[Section 1.4]{BJ}, we can switch to log-polar coordinates. Let $u_i = b_i \frac{\log|z_i|}{\log|t|}$ and $v_j = \frac{\log|w_j|}{\log|t|}$,
$\langle\underline{\kappa},\underline{u}\rangle := \sum_{i=0}^p \kappa_i u_i$, $\langle\underline{v},-\underline{\beta}+1 \rangle := \sum_{j=1}^q v_j(-\beta_j+1) $. Then,
we can write
\begin{multline*}
  \int_{X_t \cap U} f d\mu_t
  = \\
  C(\log|t|^{-1})^{p+q-d}\int_{\sigma_p \times \R^q_{\geq 0} \times Y} |t|^{2(\langle\underline{\kappa},\underline{u}\rangle+\langle\underline{v},-\underline{\beta}+1 \rangle)} \left( \int_{\Log_{t}^{-1}(u,v,y)}\phi \rho_{t,u,v,y} \right) du dv |dy|^2,
\end{multline*}
where $\phi = f \circ \Log_U$, $\rho_{t,u,v,y}$ is the Haar measure on the torsor $\Log_t^{-1}(u,v,y)$ for the (possibly disconnected) Lie-group $\{(\theta_0,\dots,\theta_p) \in (S^1)^{p+1} | e^{i\theta_0}\dots e^{i\theta_p} = 1  \} \times (S^1)^q$ and $C$ is a constant. 

First, let us try to figure out the order of magnitude of the expression on the left hand side. After re-indexing, assume that $\kappa_0 = \min_{i=1}^p \kappa_i$. Note that 
$$\int_{\sigma} |t|^{2\langle \kappa,u \rangle} du = O\left(\frac{|t|^{2\kappa_0}}{(\log|t|^{-1})^{\#\{i | \kappa_i > \kappa_0\}}}\right),$$ and for a fixed $N$ such that $\supp(f) \subset \{\sum_{i=0}^p b_iu_i = 1 \} \times [0,N]^q$, 
\begin{equation}
\int_{[0,N]^q} |t|^{\sum_{j=1}^q(-2\beta_j + 2)v_j}dv = O\left(\frac{1}{(\log|t|^{-1})^{\#\{j | \beta_j < 1\}}}\right). \label{eqnNew}
\end{equation}

Thus, we see that 
$$ \int_{U \cap X_t} (f\circ \Log_U) \chi d\mu_t = O\left(\frac{|t|^{2\kappa_0}}{(\log|t|^{-1})^{d - p - q + \#\{i | \kappa_i > \kappa_0\} + \#\{j| \beta_j < 1 \}}}\right) .$$
Note that the right hand side in the above expression goes off to $0$, unless $\kappa_0 = 0$ and $d = \#\{i | \kappa_i = 0\} + \#\{j | \beta_j = 1\}$. This corresponds exactly to the case when there exists a face $\sigma_{Y'} \subset \sigma_Y$ such that $\sigma_{Y'} \subset \Delta(\L)$ and $\sigma_{Y'}$ has dimension $d$. 

After a possible re-indexing, assume that $\kappa_0 = \dots = \kappa_{p'} = 0$ and $\kappa_i > 0$ for all $i > p'$, and $\beta_1 = \dots = \beta_q = 1$ and $\beta_j < 1$ for all $j > q'$, and $p' + q' = d$. Then, $Y' = E_1 \cap \dots \cap E_{p'} \cap B_1 \cap \dots \cap B_{q'} $.

In this case, the Poincar{\'e} residue of $\psi$ at $Y'$ is given by, 
\begin{multline*}
  \Res_{Y'}(\psi)^{1/m} = \\
  z_{p'+1}^{a_{p'+1}}\dots z_p^{a_p} w_{q'+1}^{1-\beta_{q'+1}}\dots w_q^{1-\beta_q} \frac{dz_{p'+1}}{z_{p'+1}} \wedge \dots \wedge \frac{dz_p}{z_p} \wedge \frac{dw_{q'+1}}{w_{q'+1}} \wedge \dots \wedge \frac{dw_q}{w_q} \wedge d\underline{y}  
\end{multline*}

Note that $|\Res_{Y'}(\psi)|^{2/m}$ is a finite measure on $Y'$ as $a_{i}, (1-\beta_j) > 0$ for all $i > p'$ and $j > q'$. Using the expression of $\psi_t$ in Equation \eqref{SimplifiedPsiT}, we can write 
$$
i^{n^2}\psi_t^{1/m} \wedge \overline{\psi_t}^{1/m} = \left| \frac{1}{b_0} \frac{dz_1}{z_1} \wedge \dots \wedge \frac{dz_{p'}}{z_{p'}} \wedge \frac{dw_1}{w_1} \wedge \dots \wedge \frac{dw_{q'}}{w_{q'}} \wedge \Res_{Y'}(\psi)^{1/m} \right|^2.
$$

Make a change of variables $z_i = |t|^{u_i}e^{i\theta_i}$ for $1 \leq i \leq p'$ and $w_j = |t|^{2v_j}e^{i \vartheta_j}$ for $1 \leq j \leq q'$. Writing $z' = (z_{p'+1},\dots,z_p)$ and $w' = (w_{q'+1},\dots,w_q)$, we can view $(z',w')$ as coordinates on $Y' \cap U$. Let $\tilde{\sigma} = \{(u,v) \in \R^{p+q} | \sum_{i=1}^p b_i u_i \leq 1 \}$.  Write $$S := \left\{(u,v,z',w') \in \tilde{\sigma} \times (Y' \cap U) \Bigg| \sum_{i=1}^{p'} b_i u_i + \sum_{i=p'+1}^{p} \frac{b_i \log|z_i|}{\log|t|} \leq 1  \right\}$$
and let $\textbf{1}_S$ denote its indicator function. Applying the change of variables, we get 
\begin{align*}
\frac{1}{(2\pi\log |t|)^d}\int_{U \cap X_t} (f\circ \Log_U) \chi  \left| \frac{1}{b_0} \frac{dz_1}{z_1} \wedge \ \dots \wedge \frac{dz_{p'}}{z_{p'}} \wedge \frac{dw_1}{w_1} \wedge \dots \wedge \frac{dw_{q'}}{w_{q'}} \wedge \Res_{Y'}(\psi)^{1/m} \right|^2 \\
= 
\frac{1}{b_0^2(2\pi)^d}\int_{ \tilde{\sigma} \times (S^1)^{p'} \times (S^1)^{q'} \times Y'} \sum_{z_0 | z_0^{b_0} = t/\Pi_{i=1}^{p'} z_{i}^{b_i}} f \cdot \chi \cdot \textbf{1}_S du  \ dv  \ d\theta \ d\vartheta \ |\Res_{Y'}(\psi)|^{2/m}.
\end{align*}

The integral on the right hand side is taken over $\tilde{\sigma} \times (S^1)^{p'} \times (S^1)^{q'} \times Y'$, where we view $(u,v) \in \tilde{\sigma}$, $\theta_i \in S^1$ for $1 \leq i \leq p'$,  $\vartheta_j \in S^1$ for $1 \leq j \leq q'$ and $(z',w') \in Y$. 

Let us analyze the pointwise limit of each of the factors appearing in the right hand side of the previous expression. We have that $$f\left(1 - \sum_{i=1}^{p'} b_i u_i - \sum_{i=p'+1}^{p} \frac{b_i\log|z_i|}{\log|t|}, u , \frac{\log|z'|}{\log|t|}, v, \frac{\log|w'|}{\log|t|}\right) \to 
f\left(1 - \sum_{i=1}^{p'} b_i u_i , u , \underline{0}, v, \underline{0}\right)
$$ pointwise on $\tilde{\sigma} \times Y$ as $t \to 0$.

As for $\chi$, note that $z_0 \to 0$ as $t \to 0$ for a fixed $(u,v,z',w') \in \tilde{\sigma} \times (Y' \cap U)$. So,
$$
\chi(z_0, |t|^{u}e^{i\theta},z', |t|^{v}e^{i\vartheta},w') \to \chi(0,z',0,w')
$$
as $t \to 0$. 

It is easy to check that $\textbf{1}_S \to 1$ a.e on $\tilde{\sigma} \times (Y' \cap U)$, and from our analysis in Proposition \ref{NormalizedLebesgueMeasureAnalysis}, we have that $b_{\sigma_{Y'}}^{-1} \lambda_{\sigma_{Y'}} = \frac{1}{b_0} du dv$ under the homeomorphism $\sigma_{Y'} \xrightarrow{\ \simeq \ } \tilde{\sigma}$ given by $(u_0,\dots,u_{p'},v_1,\dots,v_{q'}) \to (u_1,\dots,u_{p'},v_1,\dots,v_{q'})$.  The remaining factor of $\frac{1}{b_0}$ is taken care of the fact that the number of solutions $z_0$ to the equation $z_0^{b_0} = \frac{t}{\Pi_{i=1}^{p}z_i^{b_i}}$ is exactly $b_0$.

Using Lebesgue's dominated convergence theorem, we have the result. 
\end{proof}

The following lemma helps to `glue' to the result of the previous lemma to obtain a global version.
\begin{lem}
	\label{CompactSetFiniteVolume}
	Let $L$ be a compact subset of $(\X,B)^\hyb$. Then, $\limsup_{t \to 0} \int_{X_t \cap L} d\mu_t < \infty$. 
\end{lem}
\begin{proof}
  Without loss of generality assume that $V = \X$. Since $\{\Log_{V^\hyb}^{-1}(K) \cap \pi^{-1}(\frac{1}{2} \overline{\D}) \}_{K \subset_{cpt} \Delta(\X,B)}$ forms a compact exhaustion of $\X \cap \pi^{-1}(\frac{1}{2} \overline{\D})$, we may assume that $L = Log_{V^\hyb}^{-1}(K)$ for some compact $K \subset \Delta(\X,B)$.	

	We wish to show that $\limsup_{t \to 0}\int_{X_t} \mathbf{1}_L \circ \Log_V d\mu_t < \infty$. Let $V = \bigcup_{i \in I} U_i$ and $\{\chi_i\}_{i \in I}$ be a partition of unity on $\{U_i\}_{i \in I}$ such that $\Log_V = \sum_i \chi_i \Log_{U_i}$. It is enough to show that $\limsup_{t \to 0}\int_{U_i \cap X_t} \chi_i (\mathbf{1}_L \circ \Log_V) d\mu_t < \infty$ for all $i$. 
	
	Since $\Log_V - \Log_{U_i} = O(\frac{1}{\log|t|^{-1}})$ on the support of $\chi_i$, we can find a compactly supported continuous function $f$ on $\Delta(\X,B)$ such that $f \circ \Log_{U_i} \geq \mathbf{1}_L \circ \Log_V$ on $ (U_i \setminus D) \cap \supp(\chi_i)$. 
	
	Then, \[ \limsup_{t \to 0}\int_{U_i \cap X_t} \chi_i (\mathbf{1}_L \circ \Log_V) d\mu_t \leq \lim_{t \to 0}\int_{U_i \cap X_t} \chi_i (f \circ \Log_{U_i}) d\mu_t ,\] 
	and the right hand side exists and is finite by Lemma \ref{MainLemmaConvergenceThorem}. 
\end{proof}

We now prove the statement of Theorem \ref{ConvergenceTheorem} for functions that are pulled back from compactly-supported continuous functions on $\Delta(\X,B)$ via a global $\Log$ map.  
\begin{lem}
	Let $f$ be a continuous compactly supported function on $\Delta(\X,B)$ and let $V$ be a neighborhood of $\X_0$, and let $\Log_V$ be a global log function. Then, $\int_{X_t} (f \circ \Log_V) d\mu_t \to \int_{\Delta(\X,B)} f d\mu_0$ as $t \to 0$.
\end{lem}
\begin{proof}
	Let $V = \bigcup_{i \in I} U_i$ and let $\chi_i$ be a partition of unity on $U_i$ so that $\Log_V = \sum_i \chi_i\Log_{U_i}$.
	
	Then, we can write $\int_{X_t} (f \circ \Log_V) d\mu_t = \sum_i \int_{U_i \cap X_t} \chi_i (f \circ \Log_V) d\mu_t $. 
	It follows from Lemma \ref{CompactSetFiniteVolume} and from Proposition \ref{LogDiffersBy} that 
	\begin{equation}
	\label{eqn43}
	\lim_{t \to 0}\Big|\int_{U_i \cap X_t} \chi_i (f \circ \Log_V) d\mu_t - \int_{U_i \cap X_t} \chi_i (f \circ \Log_{U_i}) d\mu_t \Big| = 0.
	\end{equation}
	
	If $\sigma_{U_i}$ contains a maximal face of $\Delta(\L)$, it follows from Lemma \ref{MainLemmaConvergenceThorem} that 
	$$ \lim_{t \to 0} \int_{U_i \cap X_t} \chi_i (f \circ \Log_{U_i}) d\mu_t =  \left( \int_{Y_{\sigma'}}\chi_i|\Res_{Y_{\sigma'}}(\psi)|^{2/m}\right)\int_{\sigma'} f b_{\sigma'}^{-1} \lambda_{\sigma'} $$
	where $\sigma' \subset \sigma_{U_i}$ is a maximal face of $\Delta(\L)$. If $\sigma_{U_i}$ does not contain a maximal face of $\Delta(\L)$, the above limit is 0. Note that any $\sigma_{U_i}$ contains at most one maximal face $\sigma$ of $\Delta(\L)$ and this happens if and only if $Y_\sigma$ intersects $U_i$. Thus, for all $i \in I$, we have 
	\begin{equation}
	\label{eqn44}
		\lim_{t \to 0} \int_{U_i \cap X_t} \chi_i (f \circ \Log_{U_i}) d\mu_t = 
		 \sum_{\sigma \subset\Delta(\L), \dim(\sigma) = d} \left(\int_{Y_\sigma} \chi_i |\Res_{Y_\sigma}(\psi)|^{2/m} \right) b_{\sigma}^{-1} \lambda_\sigma.
	\end{equation}
	
	Combining Equations \eqref{eqn43} and \eqref{eqn44}, we are done. 
\end{proof}

Now, we are ready to prove Theorem \ref{ConvergenceTheorem}.

\begin{proof}[Proof of Theorem \ref{ConvergenceTheorem}]
	Let $f$ be a continuous compactly supported function on $(\X,B)^\hyb$. 
	Fix a global log function $\Log_V$ and let $\chi$ be a continuous function that is $1$ in a neighborhood of $\X_0$ and is supported in $\pi^{-1}(\frac{1}{2}\overline{\D})$. By replacing $f$ by \linebreak $(f|_{\Delta(\X,B)} \circ \Log_V)\cdot \chi - f$, we may assume that $f|_{\Delta(\X,B)} = 0$. 
	
	Let $K = \supp(f)$ and pick $\epsilon > 0$. Since $f$ is continuous and compactly supported, there exists $t_0 \ll 1$ such that $|f| \leq \epsilon$ on $\pi^{-1}(t\overline{\D})$. Then, $\limsup_{t \to 0} |\int_{X_t} f d\mu_t| \leq \epsilon \limsup_{t \to 0} \int_{K \cap X_t} d\mu_t$, which goes to 0 as $\epsilon \to 0$ by Lemma \ref{CompactSetFiniteVolume}. 
\end{proof}

\begin{eg}[Convergence of Haar measure on $(\P^1,0 + \infty)^\hyb$]
  \label{ExampleConvergenceHaarP1}
In the setting of Example \ref{HybridSpaceOfP1TimesD}, let $\mu_t$ denote the Haar measure on $(\P_1 \setminus \{0,\infty\}) \times \{t\}$. Then, $\frac{1}{2\pi\log|t|^{-1}}\mu_t$ weakly converges to the Lebesgue measure on $\R \simeq \Delta(\X,B)$ as measures on the hybrid space $(\X,B)^\hyb$. 
\end{eg}

More generally, we can prove a similar result for toric varieties. 

\begin{eg}[Convergence for a torus]
  \label{EgToricVarieties}
  Let $N$ be a free abelian group of rank $n$. Let $M = \Hom_{\Z}(N,\Z)$ and $T = \Spec(\C[M])$ be the associated torus. Let $Y$ be a smooth projective toric compactification of $T$ i.e.~a smooth projective toric variety associated to a regular fan in $N_{\R}$ (For example, $Y = \P^n$). Let $\omega$ be a torus invariant meromorphic 1-form on $Y$. Note that there is a canonical choice of such an $\omega$ up to a sign and $\omega$ has poles of order one along all boundary divisors. Let $D$ be the reduced divisor given by the sum of the boundary divisors. Then, $\omega \in H^0(K_Y + D)$.
  
  Consider the constant family $Y \times \D^*$ over $\D$. Then $(Y \times \D^*, D \times \D^*)$ is log smooth and consider the projective snc model $\Y = Y \times \D$ of $(Y \times \D^*,D \times \D^*)$. Then, $\Delta(\X,B)$ is canonically isomorphic to $N_\R$, with the faces given by the cones in the fan defining $Y$. Thus we have a hybrid space given by $(\Y,D \times \D)^\hyb = (T \times \D^*) \cup N_\R$. We also get a top-dimensional meromorphic form $\eta$ on $Y \times \D^*$ whose restriction to each fiber gives the measure $\omega$. Let $\mu_t$ denote the measure induced by $\omega$ on the fiber $ T \times \{t\}$ scaled by a factor of $\frac{1}{(2\pi\log|t|^{-1})^n}$.

  Applying Theorem \ref{MainThmSmoothLC} to this setting, we get that the the measures $\mu_t$ converge to the Lebesgue measure on each of the cones. The Lebesgue measures on each of the cones is exactly the Lebesgue measure on $N_\R$ (normalized by $N$) restricted to that cone. Thus, $\mu_t$ converges weakly to the Lebesgue measure on $N_\R$ as $t \to 0$.  
 \end{eg}

\section{Convergence on the limit hybrid model}
\label{SectionNonArchimedean}

The choice of a hybrid space $(\X,B)^\hyb$ depends on the choice of the model $\X$ of $X$. We construct a canonical hybrid space $(X,B)^\hyb$ that does not depend on a choice of a model. Such a space is obtained by an inverse limit $(X,B)^\hyb = \varprojlim_{\X} (\X,B)^\hyb$. Theorem \ref{LimitSkeletonIsNonArchimedean} implies that this definition matches with the definition in the introduction when $(X,D)$ is a projective and meromorphic over $D^*$.  We also explain how the space $(X,B)^\hyb$ can itself be viewed as an analytic space when $(X,D)$ is projective and meromorphic over $\D^*$.

\subsection{The limit hybrid model}
Given two models $\X',\X$ of $(X,B)$, there is always a bimeromorphic map $\X' \dashrightarrow \X$ induced by the given isomorphism with $X$ over $\D^*$. We say that $\X'$ dominates $\X$ when this bimeromorphic map extends to a morphism. More precisely, we say that  $\X'$ dominates $\X$ if we have a proper holomorphic map $\X' \to \X$ which commutes with the projection to $\D$ and extends the identity map $X \to X$.

When $\X$ and $\X'$ are proper snc models of $(X,B)$ such that $\X'$ dominates $\X$ via a map $\pi : \X' \to \X$, we also have an integral affine map $\pi_* : \Delta(\X',B) \to \Delta(\X,B)$ and also a continuous surjective map $(\X',B)^\hyb \to (\X,B)^\hyb$ as in Section 4.2 and Section 4.8 of \cite{BJ}. If $\sigma_{Y'}$ is a face of $\Delta(\X',B)$, associated to a stratum $Y'$ of $\X_0'$, by identifying the smallest stratum $Y$ that contains $\pi(Y')$. Then, $\pi_* (\sigma_{Y'}) \subset \sigma_Y$. We describe these maps in detail in the projective case in the following subsection.

The collection of all proper snc models of $(X,B)$ is a directed system. See \cite[Lemma 4.1]{BJ} for more details. 
We can then define $(X,B)^\hyb := \varprojlim_{\X} (\X,B)^\hyb$. It is easy to see that we have a projection map $(X,B)^\hyb \to \D$ such that $\pi^{-1}(\D^*) \simeq X \setminus B$, and the central fiber $(X,B)^\hyb_0$  is  $\varprojlim_{\X} \Delta(\X,B)$, where the inverse limit runs over all proper snc models $\X$ of $(X,B)$, and the inverse limit is taken in the category of topological spaces. Theorem \ref{LimitSkeletonIsNonArchimedean} tells us why this definition of $(X,B)^\hyb$ matches with the one in the introduction. 

Suppose now that $(X,B)$ is projective over $\D^*$, i.e. we can view $X$ as a closed subset of $\P^N \times \D^*$ for some $N$ such that $X$ and $B$ are cut out by polynomials whose coefficients are holomorphic on $\D^*$ and meromorphic on $\D$. Thus, we can view the coefficients of the defining equations as elements of $\Ct$. Using the same defining equations in $\P^N_{\Ct}$, we get varieties $X_{\Ct}$ and $B_{\Ct}$ over $\Spec \Ct$. A smooth projective snc model $\X$ of $(X,B)$ gives rise to an snc model $\X_{\C[[t]]}$ over $\Spec \C[[t]]$ whose generic fiber is $X_{\Ct}$ and special fiber is $\X_0$, and $\X_0 + \overline{B_{\Ct}}$ is an snc divisor in $\X_{\C[[t]]}$.
Then, we can define $\Delta(\X_{\C[[t]]},B_{\Ct})$ as the dual complex of the divisor $(\X_{\C[[t]]})_0 + B_{\Ct}$, and we have that $\Delta(\X_{\C[[t]]},B_{\Ct}) \simeq \Delta(\X,B)$. 

The following theorem, analogous to \cite[Theorem 10]{KS} \cite[Cor 3.2]{BFJ16}, realizes the central fiber $(X,B)^\hyb_0$ as a non-Archimedean space. 
\begin{thm}
	\label{LimitSkeletonIsNonArchimedean}
	We have an isomorphism $X_{\Ct}^\an \setminus B_{\Ct}^\an \simeq \varprojlim_{\X} \Delta(\X,B)$ where $(\_)^\an$ denotes the Berkovich analytification with respect to the $t$-adic norm on $\Ct$ and, the inverse limit is taken over all smooth projective snc models $(\X,B)$ of $(X,B)$.
\end{thm}

We will prove the above theorem in the following section, after setting up some preliminaries. 

\subsection{The central fiber of the limit hybrid model as a non-Archimedean space}
\label{skeleton}
For the remainder of this subsection, we assume that $X$ is a smooth proper variety over the discretely valued field $K = \Ct$, $B \subset X$ is a snc divisor and, $\X$ is a smooth proper integral scheme over $R = \C[[t]]$ along with a specified isomorphism $\X_{K} \simeq X$ such that $\X$ is an snc model of $(X,B)$ (that is, $\X_{0} + B$ is a snc divisor in $\X$). Then, $\Delta(\X,B)$ is the dual intersection complex of the divisor $\X_0+ B$. We also have a CW complex $\Delta(\X) := \Delta(\X,0)$, which can be viewed as a subcomplex of $\Delta(\X,B)$. Let $X^\an$ and $B^\an$ denote the Berkovich analytification of $X$ and $B$, respectively, with respect to the $t$-adic norm on $K$.

We have an inclusion $i_\X : \Delta(\X) \to X^\an$ and a retraction $r_\X : X^\an \to \Delta(\X)$ as constructed in \cite{MN}. 
 We would like to do a similar construction for $\Delta(\X,B)$ and $X^\an \setminus B^\an$.

Let $\X$ be a proper snc model of $(X,B)$. Then, we have an inclusion map $i_{(\X,B)} : \Delta(\X,B) \to X^\an \setminus B^\an$, which is given as follows.
Let $Y =_{\text{locally}} E_0 \cap \dots \cap E_p \cap B_1 \cap \dots \cap B_q$ denote a stratum of $\X_0$. Pick a point $(r_0,\dots,r_p,s_1,\dots,s_q) \in \sigma_{Y}$. Let $z_i$ and $w_j$ locally define $E_i$ and $B_j$ near $Y$ for $0 \leq i \leq p$ and $1 \leq j \leq q$. Then, we have an isomorphism $\widehat{\O_{\X_0,Y}} \simeq \C[[z_0,\dots,z_p,w_1,\dots,w_q]]$. Pulling back the valuation defined by $\nu(\sum_{\alpha \in \N^{p+1}, \beta \in \N^q }c_{\alpha,\beta}z^\alpha w^\beta) = \min_{c_{\alpha,\beta} \neq 0}\{\alpha\cdot \underline{r} + \beta \cdot \underline{s}\}$, we get an element of $X^\an \setminus B^\an$. It is clear that $i_{(\X,B)}$ is injective, and it follows from \cite[Prop. 3.1.4]{MN} that $i_{(\X,B)}$ is continuous. The image of $i_{(\X,B)}$ is denoted as $\Sk(\X,B)$.

We also have a continuous retraction map $r_{(\X,B)} : X^\an \setminus B^\an \to \Delta(\X,B)$, which is a left inverse to the map $i_{(\X,B)}$, defined as follows.
Since $\X$ is proper, every valuation in $X^\an \setminus B^\an$ has a center in $\X_0$. Pick $x \in X^\an \setminus B^\an$. Pick the smallest stratum $Y =_{\text{locally}} E_0 \cap \dots \cap E_p \cap B_1 \cap \dots \cap B_q $ containing $\red_\X(x)$.
Then, we define $$r_{(\X,B)}(x) = (\nu_x(E_0),\dots,\nu_x(E_p),\nu_x(B_1),\dots,\nu_x(B_q))$$ in $\sigma_Y$.

To see why $r_{(\X,B)}$ is continuous, recall that the map $X^\an \to \X_0$ taking any valuation to its center is anti-continuous (i.e. the inverse image of a closed set is open).
For any stratum $Y =_{\text{locally}} E_0 \cap \dots \cap E_p \cap B_1 \cap \dots \cap B_q $ of $\X_0$, the subset $r_{(\X,B)}^{-1}(\sigma_Y) \subset X^\an \setminus B^\an$ is a closed set as it corresponds to a subset of $X^\an$ whose center lies on an open set of $\X_0$. Therefore, it is enough to prove that $r_{(\X,B)}|_{r_{(\X,B)}^{-1}(\sigma_Y)} : r_{(\X,B)}^{-1}(\sigma_Y) \to \sigma_Y$ is continuous for all possible strata $Y$. But this is clear from the description of the map above.

We also have a continuous retraction map $\phi_{\X} : \Delta(\X,B) \to \Delta(\X)$, which we obtain from the composition.
	$$ \Delta(\X,B) \xrightarrow{i_{(\X,B)}} X^\an \setminus B^\an  \hookrightarrow X^\an  \xrightarrow{r_{\X}} \Delta(\X).$$
More explicitly, if $Y =_{\text{locally}} E_0 \cap \dots \cap E_p \cap B_1 \cap \dots \cap B_q$, let $Y' =_{\text{locally}} E_0 \cap \dots \cap E_p$ containing $Y$. Then, $\phi_{\X}(\sigma_{Y}) \subset \sigma_{Y'}$ and 
$$ \phi_{\X}(r_0, \dots, r_p,s_1,\dots,s_q) = (r_0,\dots,r_p) .$$	 

If $\X$ and $\X'$ are two proper snc models of $(X,B)$ such that $\X'$ dominates $X$, then there is a surjective map $r_{\X',\X,B} : \Delta(\X',B) \to \Delta(\X,B)$  given by 
$$\Delta(\X',B) \xhookrightarrow{i_{(\X',B)}} X^\an \setminus B^\an \xtwoheadrightarrow{r_{(\X,B)}} \Delta(\X,B).$$ The surjectivity of the map follows from  \cite[Prop. 3.17 ]{MN}. 

We have an explicit description of $r_{\X',\X,B}$ similar to \cite[Section 4.2]{BJ} as follows. Let $\rho : \X' \to \X$ denote the proper map between $\X'$ and $\X$, let $Y' = E_0' \cap \dots \cap E'_{p'} \cap B_1 \cap \dots \cap B_{q'}$ be a stratum of $\X'_0$, and let $Y = E_0 \cap \dots \cap E_p \cap B_1 \dots \cap B_q$ be the stratum of $\X_0$ containing the image of $Y'$.  Note that $q' \leq q$. Let $E_i, B_j$ be locally defined by $z_i' = 0$ and $w_j' = 0$ near $Y'$ and let $E_i$ and $B_j$ be locally defined by $z_i = 0$ and $w_j = 0$ near $Y$. Then, we can write $\rho^*(z_i) = u_i \cdot \prod_{k = 0}^{p'} (z'_k)^{c_{i,k}} $ and $\rho^*w_{j} = v_j \cdot w'_j \cdot  \prod_{k = 0}^{p'} (z'_k)^{d_{j,k}}$ for units $u_i,v_j \in \O_{\X',Y'}$ and for some $c_{i,k}, d_{j,k} \in \N$. Then, $r_{\X',\X,B}(\sigma_{Y'}) \subset \sigma_{Y}$ and is given by $$r_i = \sum_{k=0}^{p'} c_{i,k} r_k'$$ and $$s_j = s_j' + \sum_{k=0}^{p'} d_{j,k} r'_j .$$ 

It follows from the explicit description of $r_{\X',\X,B}$ that 
\begin{equation}
\label{LogWellBehavedWithRetraction}
	\Log_{\X'} \circ r_{\X',\X,B} = \Log_{\X} + O(\frac{1}{\log|t|}) 
\end{equation}
 uniformly on compacts in a neighborhood of  $\X_0$ as $t \to 0$.

\begin{prop}
	\label{MapPhiExists}
We have a commutative diagram 
$$
\begin{tikzcd}
\Delta(\X',B) \arrow[r, "\phi_{\X'}"] \arrow[d,"r_{\X',\X,B}"]& \Delta(\X') \arrow[d,"r_{\X',\X}"] \\
\Delta(\X,B) \arrow[r, "\phi_{\X}"]& \Delta(\X)
\end{tikzcd}
$$
which gives rise to a continuous map $\phi : \varprojlim_{\X} \Delta(\X,B) \to \varprojlim_{\X} \Delta(\X)$. 
\end{prop}

\begin{proof}
	To see that the diagram commutes, it enough to use the fact that $r_{\X',\X} \circ r_{\X'} = r_{\X}$ \cite[Prop. 3.1.7]{MN} and show that $ \phi_{\X} \circ r_{(\X,B)} = r_{\X}$ on $X^\an \setminus B^\an$. Pick $\nu \in X^\an \setminus B^\an$. Let $Y =_{\text{locally}} E_0 \cap \dots \cap E_p \cap B_1 \cap \dots \cap B_q$ be the minimal stratum of $\X_0 + B$ containing the center of $\nu$. Then, $$r_{(\X,B)}(\nu) = (\nu(E_0),\dots,\nu(E_p),\nu(B_1),\dots,\nu(B_q))$$ in $\sigma_Y$.
	
	 Let $Y' =_{\text{locally}} E_0 \cap \dots E_p$ be the stratum containing $Y$. Then, $Y'$ is the minimal stratum in $\X_0$ containing the center of $\nu$ and $r_{\X}(\nu) = (\nu(E_0),\dots,\nu(E_p))$ in $\sigma_{Y'}$. It follows from the description of $\phi_{\X}$ that $\phi_{\X}(r_{(\X,B)}(\nu)) = r_\X(\nu)$  
\end{proof}

\begin{prop}
	If $\X'$ is a blowup of $\X$ along a stratum $Y = E_0 \cap \dots \cap E_p \cap B_1 \cap \dots \cap B_q$, then $r_{\X',\X,B} : \Delta(\X',B) \to \Delta(\X,B)$ is a homeomorphism obtained by a subdivision.
\end{prop}
\begin{proof}
  This follows from a local blowup computation. Let $E'$ denote the exceptional divisor in $\X'$. Then, the strata of $\X'$ that map down to $Y$ are of the form $E' \cap \tilde{E}_I \cap\tilde{B}_0 \cap \dots \cap \tilde{B}_q$ and $E' \cap \tilde{E}_0 \cap \dots \cap \tilde{E}_p \cap \tilde{B}_J$, where $I$ and $J$ denote subsets of $\{0,\dots,p\}$ and $\{1,\dots,q\}$ of size $p$ and $(q-1)$ respectively and $\tilde{E}_i$ and $\tilde{B}_j$ denote the strict transforms of $E_i$ and $B_j$.

  First, let's compute the image of $\sigma_{E'}$ in $\Delta(\X,B)$. Note that $\div_{\X'}(t) = \sum_i b_i\tilde{E}_i + (\sum_{i=0}^p b_i) E'$. Let $\nu_{E'}$ denote the divisorial valuation corresponding to $\sigma_{E'}$. Then,  $$\nu_{E'}(E_i) = \nu_{E'}(\tilde{E}_i + E') = \nu_{E'}(E') = \frac{1}{\ord_{E'}(t)} = \frac{1}{\sum_{i=0}^p b_i}$$ for all $i = 0,\dots, p$.  Similarly, $\nu_{E'}(B_j) = \frac{1}{\sum_{i=0}^p b_i}$ for all $j = 1,\dots, q$. Thus, the image of $\sigma_{E'}$ in $\Delta(\X,B)$ is $\frac{1}{\sum_{i=0}^p b_i}(1,\dots,1)$.

  It is easy to check that the $\Delta(X',B) \to \Delta(\X,B)$ is a subdivision obtained by adding the vertex $\sigma_{E'}$. For example, let's compute the image of $\sigma_{Y'}$ for $Y' = E' \cap \tilde{E}_1 \cap \dots \cap \tilde{E}_p \cap \tilde{B}_1 \cap \dots \cap \tilde{B}_q$. Note that $$ \sigma_{Y'} = \left\{(x_0,\dots,x_p,y_1,\dots,y_q) \ \Big| \  \left(\sum_{i=0}^p b_i\right) x_0 + \sum_{i=1}^p b_ix_i = 1 \right\}.$$

  Suppose $\nu$ is a valuation represented by $(x_0,\dots,x_p,y_1,\dots,y_q) \in \sigma_{Y'}$. Then, $\nu(E_0) = \nu(\tilde{E}_0 + E') = \nu(E') = x_0$ and $\nu(E_i) = \nu(\tilde{E}_i + E') = x_i + x_0$ for $i = 1,\dots,p$. Similarly, $\nu(B_j) = y_j + x_0$ for $j=1,\dots,q$.
  
  Thus, we see that $r_{\X',\X,B}|_{\sigma_{Y'}}$ is given by
  $$ (x_0,\dots,x_p,y_1,\dots,y_q) \mapsto (x_0,x_1+x_0,\dots,x_p+x_0,y_1+x_0,\dots,y_q+x_0) $$
\end{proof}

In general, the map $\Delta(\X') \to \Delta(\X)$ is not a homeomorphism, as illustrated by the following example. 

\begin{eg}[Blowup of $\P^1 \times \D$]
  Let the notation be the same as in Example \ref{P1DualComplex}. Let $E_0 = \P^1 \times \{0\}$, $B_1 = \{0\} \times \D$, $B_2 = \{\infty\} \times \D$. Let $\X'$ denote the blowup of $\X$ at $E_0 \cap B_1$ and let $\X''$ denote the blowup of $\X$ at some point in $E_0$ that is different from $0$ and $\infty$. Then $\Delta(\X',B)$ is obtained from $\Delta(\X,B)$ by adding a vertex along the ray $E_0 \cap B_1$ and $\Delta(\X'',B)$ is obtained from $\Delta(\X,B)$ by adding an extra vertex and joining it to $\sigma_{E_0}$.

  The retraction $r_{\X',\X,B} : \Delta(\X',B) \to \Delta(\X,B)$ is an isomorphism, while  $r_{\X'',\X,B} : \Delta(\X'',B) \to \Delta(\X,B)$ is given by collapsing the newly added edge and vertex to $\sigma_{E_0}$.  
  \begin{figure}
    \label{FigDualComplexAfterBlowup}
    \begin{minipage}{0.45\linewidth}
    \begin{tikzpicture}
      	\draw [<->] (-3,0)-- (0,0) -- (3,0);
	\draw [fill] (0,0) circle [radius = 0.05];
        \draw [fill] (1,0) circle [radius = 0.05];
        \draw [fill] (1,-0.5) circle [radius = 0];

      \end{tikzpicture}
      \caption*{(a)}
  \end{minipage}
  \hspace{0.05\linewidth}
  \begin{minipage}{0.45\linewidth}
    \begin{tikzpicture}
      	\draw [<->] (-3,0)-- (0,0) -- (3,0);
	\draw [fill] (0,0) circle [radius = 0.05];
        \draw (0,0) -- (1,-0.5);
        \draw [fill] (1,-0.5) circle [radius = 0.05];

      \end{tikzpicture}
      \caption*{(b)}
  \end{minipage}
  \caption{The dual complexes of the $(\P^1 \times \D, \{0\} \times \D + \{\infty\} \times \D)$} after  (a) blowing up at $(0,0)$, and (b) blowing up at (1,0) 
  \end{figure}
\end{eg}

\begin{lem}
	\label{EverythingLiesInBoundedPiece}
	Let $\X$ be a proper snc model of $(X,B)$ and let $K \subset \Delta(\X,B)$ be a compact set. Then there exists a proper snc model $\X'$ of $(X,B)$ dominating $\X$ such that $r_{\X',\X,B}^{-1}(K) \subset \Delta(\X')$. 
\end{lem}
\begin{proof}
  For a valuation $\nu \in X^\an$ and a divisor $D \subset \X$ not contained in $\overline{\{\ker{\nu}\}}$, set $\nu(D) := \nu(f)$, where $f$ defines $D$ in an open neighborhood of the $\red_{\X}(\nu)$.
  
  Since it is enough to prove the result for some neighborhood of all points in $K$, we may assume without loss of generality that there exists an irreducible component $E$ of $\X_0$ and an $\epsilon > 0$ such that $\nu(E) \geq \epsilon$ for all $\nu \in K$. Let $B_1,\dots,B_q$ be the irreducible components of $B$ containing the centers of all $\nu \in K$.   It is enough to show that there exists a smooth proper snc model $\X'$ of $(X,B)$ such that $\red_{\X'}(\nu')$ is not contained in the closures of $B_1,\dots,B_q$ in $\X'$ for all $\nu' \in \Delta(\X,B)$ such that $r_{\X',\X,B}(\nu') \in K$. Note that if $q = 0$, we are done. We will prove the result by induction on $q$.

  Pick $N > 0$ large enough so that $N\nu(E_1) \geq \nu(B_1)$ for all $\nu \in K$.
  Let $\I_{E}$ and $\I_{B_1}$ be the ideal sheaf defining $E$ and $B_1$ respectively. Let $\widetilde{\X}$ be the blowup of $\X$ along the ideal sheaf $\I_{E}^N + \I_{B_1}$. Then, $\widetilde{\X}$ is a proper model of $(X,B)$ although it may not necessarily be smooth. Pick $\nu \in K$ and let $U$ be an affine open neighborhood of $\red_{\X}(\nu)$. If $E$ is defined by $z = 0$ and $B_1$ is defined by $w_1 = 0$ on $U$, then $\widetilde{U} = \Spec \O_{\X}(U)[\frac{z^N}{w_1}]$ is a chart of the blowup.
  Let $\X'$ be a resolution of singularities of $\widetilde{\X}$ such that $\X'$ is a proper snc model for $(X,B)$. Pick $\nu' \in \Delta(\X,B)$ such that $r_{\X',\X,B}(\nu') = \nu$. Then, $\nu'(\frac{z^N}{w_1}) = \nu(\frac{z^N}{w_1}) \geq 0$. Thus, the center of $\nu'$ in $\widetilde{\X}$ is contained in $\widetilde{U}$. But since $\widetilde{U}$ misses the strict transform of $B_1$, and thus the center of $\nu'$ in $\X'$ is not contained in $B_1$.
  Thus, the irreducible components of $B$ containing the centers of any valuations $\nu' \in r_{\X',\X,B}^{-1}(K)$ can only be $B_2,\dots,B_q$, and thus we are done by induction. 	

\end{proof}

\begin{cor}
\label{CenterAvoidsB}
Let $(\nu_{\X})_{\X} \in \varprojlim_{\X} \Delta(\X,B)$. Given a smooth proper snc model $\X$ of $(X,B)$, there exists a smooth proper snc model $\X'$ of $(X,B)$ dominating $(\X,B)$ such that the center of $\nu_{\X'}$ in $\X'$ does not intersect $B$. 
\end{cor}

\begin{proof}
This easily follows Lemma \ref{EverythingLiesInBoundedPiece}. Once we find a model $\X'$ of $\X$ such that $\red_{\X'}(\nu_{\X'})$ is not contained in the closure of $B$, we can further blowup to assume that the two become disjoint. 
 \end{proof}

\begin{prop}
	\label{PhiIsInjective}
	The map $\phi : \varprojlim_{\X} \Delta(\X,B) \to \varprojlim_{\X} \Delta(\X)$ is open and injective, where $\X$ ranges over all proper snc models $\X$ of $(X,B)$.
\end{prop}
\begin{proof}
	Pick $(\nu_\X)_{\X}, (\nu'_\X)_{\X}$ be two distinct elements in $\varprojlim_{\X} \Delta(\X,B)$. Let $\X$ be a proper snc model of $(X,B)$ such that $\nu_{\X} \neq \nu_{\X}'$ in $\Delta(\X,B)$. From Corollary \ref{CenterAvoidsB}, we can find a model $\Delta(\X'B)$ such that $\phi_{\X'}(\nu_{\X'}) = \nu_{\X'}$ and $\phi_{\X'}(\nu'_{\X'}) = \nu'_{\X'}$. Note that $\nu_{\X'} \neq \nu'_{\X'}$ as $r_{\X',\X,B}(\nu_{\X'}) \neq r_{\X',\X,B}(\nu'_{\X'})$. Thus, $\phi$ is injective.
	
	To see that $\phi$ is open, it is enough to show that $$\phi(\{(\nu_{\X})_{\X} \in \varprojlim_{\X} \Delta(\X,B) | \nu_{\Y} \in U \})$$ is open for a model $\Y$ of $(X,B)$ and an open set $U \subset \Delta(\Y,B)$. We may also assume that $U$ has compact closure. Using Lemma \ref{CenterAvoidsB}, we can find a model $\Y'$ such that  $U' := r_{\Y',\Y,B}^{-1}(U) \subset \Delta(\Y')$. Then, it is easy to check that $$\phi(\{(\nu_{\X})_{\X} \in \varprojlim_{\X} \Delta(\X,B) | \nu_{\Y} \in U \}) = \{ (\nu_{\X})_{\X} \in \varprojlim_{\X}\Delta(\X) | \nu_{\Y'} \in U' \}$$

\end{proof}

To prove Theorem \ref{LimitSkeletonIsNonArchimedean}, we exploit the isomorphism $X^\an \xrightarrow{\simeq} \varprojlim_{\X} \Delta(\X)$ (see \cite[Theorem 10]{KS}, \cite[Cor. 3.2]{BFJ16}). 
\begin{remark}
	 The homeomorphism $X^\an \xrightarrow{\simeq} \varprojlim_{\X} \Delta(\X)$ in \cite[Theorem 10]{KS} is stated when the inverse limit runs over all smooth proper models $\X$ of $X$ such that the central fiber $\X_0$ is an snc divisor. However, we may as well take the inverse limit over all smooth proper models $\X$ such that $\X_0 + B$ is an snc divisor, because such models form a cofinal system.
\end{remark}

\begin{proof}[Proof of Theorem \ref{LimitSkeletonIsNonArchimedean}]
We obtain a map $r : X^\an \setminus B^\an \to \varprojlim_{\X} \Delta(\X,B)$ by considering the inverse limit over the retraction map $r_{(\X,B)} : X^\an \setminus B^\an \to \Delta(\X,B)$. 

Observe that we have the commutative diagram where the bottom map is a homeomorphism.
$$
\begin{tikzcd}
X^\an \setminus B^\an \ar[r,"r_{(X,B)}"] \ar[d, "i", hook] & \varprojlim_{\X} \Delta(\X,B) \ar[d,hook, "\phi"] \\
X^\an \ar[r, "r_X", "\simeq"'] & \varprojlim_{\X} \Delta(\X)
\end{tikzcd}
$$

Therefore, it is enough to show that the image of $B^\an$ in $\varprojlim_{\X} \Delta(\X)$ does not intersect with the image of $\phi$. Let $(\nu_{\X})_{\X}$ be an element of $\varprojlim_{\X} \Delta(\X,B)$ and let $\nu := r_{X}^{-1}{(\phi((\nu_\X)_\X))}$. Without loss of generality, assume to the contrary that $\nu \in B_1^\an$.

Using Corollary \ref{CenterAvoidsB}, we can find a model $\X$ such that the center of $\nu_\X$ in $\X$ does not intersect $B$. Then, $\phi_{\X}(\nu_\X) = \nu_\X$. We also have that $r_\X(\nu) = \phi_\X(\nu_\X) = \nu_\X$ and the center of $\nu$ in $\X$ is contained in the center of $\nu_\X$ in $\X$. But the center of $\nu$ is contained in the closure of $B_1$, which is a contradiction.
\end{proof}

\subsection{The limit hybrid space as an analytic space}
In this section, for any $0 < r < 1$, we realize $(X,B)^\hyb_r := (X,B)^\hyb|_{\overline{r\D}}$ as the analytification of a scheme over a Banach ring, $A_r$. 

As in \cite{Ber09}, consider the Banach ring $$A_r = \{ \sum_{i \in \mathbb{Z}} c_i t^i | c_i \in \C \text{ and }\sum_{i \in \mathbb{Z}}||c_i||_\hyb r^i < \infty \},$$ where $||c_i||_{\hyb} = \max\{|c_i|, 1\}$ if $c_i \neq 0$ and $||0||_\hyb = 0$. Then, its Berkovich spectrum $\M(A_r)$ is homeomorphic to $r\overline{\D}$. For more details, see \cite{Ber09} \cite[Appendix 1]{BJ}.
Note that any function that is holomorphic in open neighborhood of $r\overline{\D} \setminus \{0\}$ and meromorphic at 0 gives an element of $A_r$. 

Given a projective family $X \to \D^*$, we can think of $X$ as a finite scheme over $\Spec A_r$ because the coefficients of the homogeneous equations cutting out $X$ in $\P^N \times \D^*$ can be viewed as elements of $A_r$. We denote this scheme as $X_{A_r}$. Similarly, we get $B_{A_r} \subset X_{A_r}$. Let $(\_)^{\An}$ denote the Berkovich analytification functor on the category of finite type schemes over $\Spec A_r$. The map $X_{A_r} \setminus B_{A_r} \to \Spec A_r$ gives rise to the canonical map $X_{A_r}^{\An} \setminus B_{A_r}^\An \to \M(A_r) \simeq r\overline{\D}$. The following proposition tells us how this analytic space is related to $(X,B)^\hyb$. 

\begin{prop}
  \label{PropHybridSpaceIsAnalytification}
	We have a homeomorphism $X_{A_r}^\An \setminus B_{A_r}^\An \xrightarrow{\simeq} (X,B)^\hyb_r$ as spaces over $r\overline{\D}$. 
\end{prop}
\begin{proof}
	Let $\pi_r : (X_{A_r} \setminus B_{A_r}) \to r\overline{\D} \simeq \M(A_r)$ be the canonical projection map. From \cite[Lemma A.6]{BJ} we have the following homeomorphisms: 
	$$ \pi_r^{-1}(r\overline{\D}^*) \simeq (X \setminus B)|_{r\overline{\D}^*} \text{ and } \pi_r^{-1}(0) \simeq (X_{\Ct}^\an \setminus B_{\Ct}^\an).$$
	Moreover, the first homeomorphism is compatible with the projections to $r\overline{\D}^*$.  
	
	The above homeomorphisms let us define a bijection $X_{A_r}^\An \setminus B_{A_r}^\An \to (X,B)^\hyb_r$. It remains to check that this map is continuous. To do this, first note that we have an embedding $(X, B)^\hyb \hookrightarrow X^\hyb $, where  $X^\hyb := \varprojlim_{\X} \X^\hyb$, given by the canonical inclusion over $\D^*$ and by Proposition \ref{MapPhiExists} over the central fiber. We also have a homeomorphism $X_{A_r}^\An \to X^\hyb_r$ as topological spaces over $r\overline{\D}$ \cite[Prop. 4.12]{BJ}. It is straightforward to check that the following diagram of topological spaces over $r\overline{\D}$ commutes. 	
	$$
	\begin{tikzcd}
	X_{A_r}^\An \setminus B_{A_r}^\An \ar[r] \ar[d, hook] &  (X,B)^\hyb_r  \ar[d, hook] \\
	X_{A_r}^\An \ar[r, "\simeq"] & X^\hyb_r
	\end{tikzcd}
	$$
	
	Since the map at the bottom is a homeomorphism and the top map is a bijection, the top map is also a homeomorphism. 
\end{proof}

Now, we can define the hybrid space associated to a (not necessarily smooth) pair $(X,B)$ over $\D^*$ as $(X,B)^\hyb  := X_{A_r}^\An \setminus B^\An_{A_r}$. Proposition \ref{PropHybridSpaceIsAnalytification} tells us that this matches with the previous definition. 

\subsection{Convergence on limit hybrid model}
The convergence described in Theorem \ref{ConvergenceTheorem} depends on the choice of a model $(\X,B)$ of $(X,B)$. We would like to remedy this by describing the convergence on $(X,B)^\hyb$, which does not require choosing a model. Given two models $(\X,B)$ and $(\X',B)$ of $(X,B)$ with $(\X',B)$ dominating $(\X,B)$ via $\rho : \X' \to \X$, a line bundle $\L$ on $\X$ extending $K_{X/\D^*} + B$, a generating section $\psi$ extending $\eta$ and a proper, we can get a line bundle $\L' = \rho^*\L$ on $\X'$ extending $K_{X/\D^*} + B$ and a section $\psi'  = \rho^* \psi$ extending $\eta$. Applying Theorem $\ref{ConvergenceTheorem}$ to both $\X$ and $\X'$, we get measures $\mu_0^{\X}$ and $\mu_0^{\X'}$ on $\Delta(\X,B)$ and $\Delta(\X',B)$ respectively. It follows from Lemma \ref{MainLemmaConvergenceThorem} and Equation \eqref{LogWellBehavedWithRetraction} that $\mu_0^{\X}$ is just the push-forward of the measure $\mu_0^{\X'}$ under the map $r_{\X',\X,B}$. 

Thus, we get a compatible system of measure $\mu_0^{\X'}$ on all models $(\X',B)$ dominating a fixed model $(\X,B)$. This gives rise to a measure on $\mu_0$ on $(X,B)_0^\hyb$, and thus we get the following convergence theorem.

\begin{thm}
  Let $(X,B)$ be a log smooth pair over $\D^*$ such that $K_{X/\D^*}+B \sim_{\Q} 0$ and let $\eta \in H^0(X,m(K_{X/\D^*} + B))$ have an analytic singularity at $t = 0$ (i.e.~there exists a model $\X$ of $(X,B)$, a line bundle $\L$ extending $K_{\X/\D^*}+B$ and $\psi \in H^0(\X,m\L)$ extending $\eta$). Then, there exists $\kappa_{\min} \in \Q$ and $d \in \N$ such that the measure $\mu_t = \frac{i \eta_t \wedge \overline{\eta_t}}{|t|^{2\kappa_{\min}}(2\pi\log|t|^{-1})^d}$ converges weakly to a measure $\mu_0$ on $(X,B)^\hyb$.

  Moreover if we fix a model $\X$, a line bundle $\L$ and a section $\psi \in H^0(\X,m\L)$ extending $\eta$, then $\mu_0$ is supported on $\Delta(\L) \subset \Delta(\X,B) \subset X^\an \setminus B^\an$, and $d, \kappa_\min$ and $\mu_0$ has the same description as in Section \ref{TheConvergenceTheoremSection}. 
\end{thm}

\begin{eg}
  Following up Example \ref{ExampleConvergenceHaarP1}, we see that the Haar
  measures on $\P^1$ converges to the Lebesgue measure on $\R$, which can be thought of as the unique line joining the type 1 points corresponding $0$ and $\infty$ in $(\P^1_{\Ct})^\an$. More generally, we could take $B_t$ is given by $p(t), q(t)$ for distinct functions $p,q$ which are meromorphic on $\D$ and holomorphic on $\D^*$. Then, there exists an isomorphism of pairs $(\P^1 \times \D^*, [p(t)] + [q(t)]) \simeq (\P^1 \times \D^*, [0] \times \D^* + [\infty] \times \D^*)$. This extends to an isomorphism $(\P_{\Ct}^1)^\an \setminus \{p, q\} \simeq (\P_{\Ct}^1)^\an \setminus \{ 0, \infty \}$, where $p, q$ denote the type 1 points corresponding to $p(t)$ and $q(t)$. Thus, as $t \to 0$, the Haar measure on $\P^1 \setminus \{ p(t), q(t) \}$ converges to the Lebesgue measure on the unique line joining the points $p$ and $q$ in $(P^1_\Ct)^\an \setminus \{\tilde{p}, \tilde{q}\}$. 
\end{eg}

\begin{eg}
  Generalizing the above example, let $X = \P^1 \times \D^*$ denote the constant family. Let $B = \{ z^2 + a_1z + a_2 = 0  \} \subset \P^1 \times \D^*$, where $z$ denotes the coordinate on $\P^1$ and $a_1, a_2$ are functions that are meromorphic on $\D$ and holomorphic on $\D$. Then, $(X,B)$ is log-Calabi-Yau. Also assume that the polynomial $z^2 + a_1 z + a_2 \in \Ct[z]$ is irreducible.

  Fix a square root $u = \sqrt{t}$ and consider the field extension $\Ct \to \Cu$. This corresponds to a degree two map $\D^* \to \D^*$ given by $u \to u^2$. Let $Y$ denote the fiber product of $X \times_{\D^*} \D^*$.
  The polynomial $z^2 + a_1 z + a_2 \in \Ct[z]$ splits into factors $(z - p)(z - q)$ in $\Cu[z]$. By the previous example, as $u \to 0$, the Haar measure on $\P^1 \setminus \{ p(u), q(u) \}$ converges to the Lebesgue measure on the line joining $p$ and $q$ in $(\P^1_\Cu)^\an \setminus \{ p,q \}$. Call this measure $\widetilde{\mu_0}$. We have a map $(\P^1_\Cu)^\an \setminus \{ p,q \} \to (\P^1_\Ct)^\an \setminus B^\an$.
  
  To understand the convergence of the Haar measure on $\P^1 \setminus B_t$, note that $\P^1 \setminus B_t \simeq \P^1 \setminus \{ p(u), q(u) \}$. Thus, as $t \to 0$ the Haar measure on $\P^1 \setminus B_t$ converges to the pushforward of $\widetilde{\mu_0}$  to $(\P^1_\Ct)^\an \setminus B^\an$.    
\end{eg}

\begin{eg}
  \label{EgToricVariety2}
Following up on \ref{EgToricVarieties}, we get that the (scaled) Haar measure on the constant family of tori $T = N \otimes \C^*$ converges to the Lebesgue measure on $\R^n$. For any smooth projective toric compactification $Y$ of $T$ with boundary divisor $D$,  the image of $\Delta(Y,D) \subset T_\Ct^\an$ coincides with the image of $N_R \hookrightarrow T_\Ct^\an$ given by sending $\sum n_i \otimes a_i \in N_\R$ to the seminorm $|\sum_j a_j \chi^m_j| = \max_j \{ |a_j| e^{\sum_{i}r_i \langle m_j , n_i \rangle} \}$.  
  \end{eg}

\subsection{Convergence for Log-Canonical Pairs $(X,B)$}
In this subsection, we drop the assumption that $X$ is smooth and prove Theorem \ref{MainTheoremConvergenceLogCanonical}. Suppose that $X$ is a normal projective family of analytic varieties over $\D^*$ and $B$ is $\Q$-divisor in $X$ such that $(X,B)$ is  a sub-log-canonical, log-Calabi-Yau pair with analytical singularities at $0$. Let $\eta \in H^0(X, m(K_{X/\D^*} + B))$ be a generating section. 

Let $\pi: (Y,B') \to (X,B)$ be a log resolution of singularities. Here, $B' = \pi^*(B) + E$, where $E$ is the exceptional divisor of $\pi$. Note that $K_{Y/\D^*}+B' = \pi^*(K_{X/\D^*}+B)  \sim_{\Q}0 $. Therefore, the pair $(Y',B')$ is log-smooth, sub-log-canonical and log-Calabi-Yau. Let $\eta' \in H^0(Y, m(K_{Y/\D^*}+B'))$ denote the section $\eta' = \pi^*(\eta)$. Applying Theorem \ref{MainThmSmoothLC} to $Y$, we get that there exist $\kappa_\min \in \Q, d \in \N_{+}$ such the measures $\mu_t' = \frac{i^{n^2}\eta'_t \wedge \overline{\eta_t'}}{(2\pi\log|t|^{-1})^d |t|^{2\kappa_\min}}$ converge weakly to a measure $\mu_0'$ on the space $(Y,B'
)^\hyb_r$ for any $ 0 < r < 1$. 

Let $\pi^{\An}_{A_r} : (Y,B')_r \to (X,B)_r$ denote the continuous map induced by $\pi$ on the analytification. Then, it follows from a change of variable formula that $\mu_t := (\pi^{\An}_{A_r})_*(\mu'_t) = \frac{i^{n^2}\eta_t \wedge \overline{\eta_t}}{(2\pi\log|t|^{-1})^d |t|^{2\kappa_\min}}$. From the continuity of $\pi^{\An}_{A_r}$, it follows that $\mu_t \to (\pi^\An_{A_r})_*(\mu'_0)$, which finishes the proof of Theorem \ref{MainTheoremConvergenceLogCanonical}. 

\bibliography{biblio.bib}
\bibliographystyle{alpha}

\end{document}